\documentclass[11pt,a4paper,english]{article}

\newcommand{\beq}{\begin{equation}}
\newcommand{\eeq}{\end{equation}}
\newcommand{\bea}{\begin{eqnarray}}
\newcommand{\eea}{\end{eqnarray}}
\newcommand{\beas}{\begin{eqnarray*}}
\newcommand{\eeas}{\end{eqnarray*}}

\usepackage[latin1]{inputenc}
\usepackage{babel}
\usepackage[centertags]{amsmath}
\usepackage{amsfonts}
\usepackage{amssymb}
\usepackage{color}
\usepackage{amsthm}
\usepackage{newlfont}
\usepackage{graphicx}
\usepackage{dsfont}
\usepackage{geometry}
\usepackage{mathrsfs}
\usepackage{authblk}
\usepackage{verbatim}

\newtheorem{theorem}{Theorem}[section]
\newtheorem{lemma}[theorem]{Lemma}
\newtheorem{coroll}[theorem]{Corollary}
\newtheorem{prop}[theorem]{Proposition}
\newtheorem{definition}[theorem]{Definition}
\newtheorem{remark}[theorem]{Remark}
\newtheorem{ass}[theorem]{Assumption}

\newcommand{\pd}[2]{\frac{\partial#1}{\partial#2}}

\def\RR{\mathbb R}
\def\EE{\mathsf E}
\def\PP{\mathsf P}
\def\cC{\mathcal C}

\def\cF{{\cal F}}

\def\cS{{\cal S}}

\makeatletter  
\@addtoreset{equation}{section}
\def\theequation{\arabic{section}.\arabic{equation}}
\makeatother  
\begin{document}

\title{A solvable two-dimensional singular stochastic control problem with non convex costs\footnote{ The first and the third authors were supported by EPSRC grant EP/K00557X/1; financial support by the German Research Foundation (DFG) via grant Ri--1128--4--2 is gratefully acknowledged by the second author.}}

\author{Tiziano De Angelis\thanks{School of Mathematics, The University of Leeds, Woodhouse Lane, Leeds LS2 9JT, United Kingdom; \texttt{t.deangelis@leeds.ac.uk}}\:\:\:\:Giorgio Ferrari\thanks{Center for Mathematical Economics, Bielefeld University, Universit\"atsstrasse 25, D-33615 Bielefeld, Germany; \texttt{giorgio.ferrari@uni-bielefeld.de}}\:\:\:\:John Moriarty\thanks{School of Mathematical Sciences, Queen Mary University of London, Mile End Road, London E1 4NS, United Kingdom; \texttt{j.moriarty@qmul.ac.uk}}}

\date{\today}
\maketitle

\textbf{Abstract.} In this paper we provide a complete theoretical analysis of a two-dimensional degenerate non convex singular stochastic control problem. The optimisation is motivated by a storage-consumption model in an electricity market, and features a stochastic real-valued spot price modelled by Brownian motion. We find analytical expressions for the value function, the optimal control and the boundaries of the \emph{action} and \emph{inaction} regions. The optimal policy is characterised in terms of two monotone and discontinuous repelling free boundaries, although part of one boundary is constant and the smooth fit condition holds there.
\medskip


\smallskip

{\textbf{MSC2010 subsject classification}}: 91B70, 93E20, 60G40, 49L20.




\section{Introduction}
\label{Introduction}

Consider the following problem introduced in \cite{DeAFeMo14}: a firm purchases electricity over time at a stochastic spot price $(X_t)_{t \geq 0}$ for the purpose of storage in a battery. The battery must be full at a random terminal time $\tau$, and any deficit leads to a terminal cost given by the product of a convex function $\Phi$ of the undersupply and the terminal spot price $X_{\tau}$. The terminal cost accounts for the use of a quicker but less efficient charging method at the time $\tau$ of demand, while the restriction to purchasing is interpreted as the firm not having necessary approval to sell electricity to the grid.

Taking $X$ as a real-valued Markov process carried by a complete probability space $(\Omega,\cF,\PP)$, and letting $\tau$ be independent of $X$ and exponentially distributed with parameter $\lambda>0$, it is shown in Appendix A of \cite{DeAFeMo14} that this optimal charging problem is equivalent to solving 
\beq
\label{problem-intro}
U(x,c)=\inf_{\nu}\EE\bigg[\int_0^{\infty}e^{-\lambda t}\lambda X^x_t\Phi(c +\nu_{t})dt + \int^{\infty}_0 {e^{-\lambda t}X^x_t\,d{\nu}_t} \bigg],\qquad (x,c) \in \RR \times [0,1].
\eeq
Here $\Phi$ is taken to be a strictly convex, twice continuously differentiable, decreasing function and the infimum is taken over a suitable class of nondecreasing controls $\nu$ such that $c +\nu_t \leq 1$, $\PP$-a.s.\ for all $t \geq 0$. The control $\nu_t$ is the cumulative amount of energy purchased up to time $t$ and $c +\nu_t$ represents the inventory level at time $t$ of the battery whose inventory level is $c$ at time $0$. The \emph{finite fuel} constraint $c + \nu_t \leq 1$, $c\in [0,1]$, $\PP$-a.s.\ for all $t\geq0$, reflects the fact that the battery has limited total capacity.

Certain deregulated electricity markets with renewable generation exhibit periods of negative electricity price, due to the requirement to balance real-time supply and demand. Such negative prices are understood to arise from a combination of the priority given to highly variable renewable generation, together with the short-term relative inflexibility of traditional thermal generation units \cite{Genoese}, \cite{Epex}. In order to capture this feature, which is unusual in other areas of mathematical finance, we assume a one-dimensional spot price $X$ taking negative values with positive probability. In \cite{DeAFeMo14} $X$ is an Ornstein-Uhlenbeck (OU) process. In the present paper, with the aim of a full theoretical investigation, we take a more canonical example letting $X$ be a Brownian motion and we completely solve problem \eqref{problem-intro}.

From the mathematical point of view, \eqref{problem-intro} falls into the class of singular stochastic control (SSC) problems. The associated Hamilton-Jacobi-Bellman (HJB) equation is formulated as a two-dimensional degenerate variational problem with a state-dependent gradient constraint. The problem is degenerate because the control acts in a direction of the state space which is orthogonal to the diffusion. It is worth mentioning that explicit solutions of problems with state-dependent gradient constraints are relatively rare in the literature (a recent contribution is \cite{GuoZervos15}) when compared to problems with constant constraints on the gradient and one or two-dimensional state space (see for example \cite{Alvarez01} and \cite{MehriZervos} amongst others).

As also noted in \cite{DeAFeMo14}, a key peculiarity of our problem is that the total expected cost functional which we want to minimise in \eqref{problem-intro} is \emph{non convex} with respect to the control variable $\nu$. In particular, by recalling that $X$ is real-valued and simply writing it as the difference of its positive and negative part, it is easy to see that the cost functional in \eqref{problem-intro} can be written as a d.c.\ functional, i.e.\ as the difference of two functionals convex with respect to $\nu$ (see \cite{HorstPardalosThoai} or \cite{HorstThoai} for references on d.c.\ functions). SSC problems which are convex with respect to $\nu$ are of particular interest since they typically have optimal controls of {\em reflecting} type, leading in turn to a certain {\em differential} connection to problems of optimal stopping (OS), see for example \cite{ElKK88} and \cite{KaratzasShreve84}. Clearly, however, the d.c.\ property of the functional in \eqref{problem-intro} means that problem \eqref{problem-intro} does not fall directly into this setting. Indeed the study in \cite{DeAFeMo14}, where the uncontrolled process $X$ is of OU type, reveals how the non-convexity of the cost criterion impacts in a complex way on the structure of the optimal control and on the connection between SSC problems and OS ones. It is shown in \cite{DeAFeMo14} that while connections to OS do hold for problem \eqref{problem-intro}, they may or may not be of differential type depending on parameter values and the initial inventory level $c$. This suggests that the solutions of two-dimensional degenerate problems of this kind are complex and should be considered case by case.

In particular, the analysis in \cite{DeAFeMo14} identifies three regimes, two of which are solved and the third of which is left as an open problem under the OU dynamics. Here we aim at a complete solution of \eqref{problem-intro} and address the third regime of \cite{DeAFeMo14} in the Brownian case. Such a complete solution also gives some insight in the open case of \cite{DeAFeMo14} since Brownian motion is a special case of OU with null rate of mean reversion. The geometric methodology we employ in this paper (see Figures \ref{fig:H1} and \ref{fig:H2}) is a significant departure from that in \cite{DeAFeMo14}. In Section \ref{sec:auxiliary} below we rely on the characterisation via concavity of excessive functions for Brownian motion introduced in \cite{Dynkin}, Chapter 3 (later expanded in \cite{DayKar}) to study a parameterised family of OS problems. It is thanks to this characterisation that we succeed in obtaining the necessary monotonicity and regularity results for the optimal boundaries of the action region associated to \eqref{problem-intro} (i.e.\ the region in which it is profitable to exert control). In contrast to the OU case, the Laplace transforms of the hitting times of Brownian motion are available in closed form and it is this feature which ultimately enables the method of the present paper.

We show that the action region of problem \eqref{problem-intro} is disconnected. It is characterised in terms of two boundaries which we denote below by $c \mapsto \hat{\beta}(c)$ and $c \mapsto \hat{\gamma}(c)$ which are discontinuous, the former being non-increasing everywhere but at a single jump and the latter being non-decreasing with a vertical asymptote (see Fig.~\ref{fig:1}). Through a verification argument we are able to show that {the optimal} control always acts by inducing {discontinuities} in the state process. 
The boundaries $\hat{\beta}$ and $\hat{\gamma}$ are therefore \emph{repelling} (in the terminology of \cite{DZ} or \cite{KOWZ}). However, in contrast with most known examples of repelling boundaries, if the optimally controlled process hits the upper boundary $\hat{\beta}$ the controller does not immediately exercise all available control but, rather, causes the inventory level to jump to a \emph{critical level} $\hat{c}\in(0,1)$ (which coincides with the point of discontinuity of the upper boundary $c \mapsto \hat \beta(c)$). After this jump the optimally controlled process continues to {evolve} until hitting the lower boundary $\hat{\gamma}$ where all the remaining control is spent to fill the inventory (the upper boundary is then formally infinite; for details see Sections \ref{construction} and \ref{verification}).

The present solution does in part display a differential connection between SSC and OS. In particular, when the initial inventory level $c$ is strictly larger than $\hat{c}$ there is a single lower boundary $\hat{\gamma}$ which is constant.
Moreover $U_c$ coincides with the value function of an associated optimal stopping problem on $\RR\times(\hat{c},1]$ and the so-called {\em smooth fit} condition holds at $\hat{\gamma}$ (for $c>\hat{c}$) in the sense that $U_{xc}$ is continuous across it. This constant boundary can therefore be considered {\em discontinuously} reflecting. That is, it may be viewed as a limiting case of the more canonical strictly decreasing reflecting boundaries. On the other hand, when the initial inventory level $c$ is smaller than the critical value $\hat{c}$ the control problem is more challenging due to the presence of two moving boundaries, which we identify in Section \ref{sec:auxiliary} with the optimal boundaries of a {\em family} of auxiliary OS problems. In this case it can easily be verified that $U_{xc}$ is discontinuous across the optimal boundaries so that the smooth fit condition breaks down, and there is no differential connection to OS.

Smooth fit is one of the most studied features of OS and SSC theory and it is \emph{per se} interesting to understand why it breaks down. It is known for example (see \cite{Pe07}) that diffusions whose scale function is not continuously differentiable may induce a lack of smooth fit in OS problems with arbitrarily regular objective functionals. On the other hand when the scale function is $C^1$ Guo and Tomecek \cite{GuoTom09} provide necessary and sufficient conditions for the existence of smooth fit in two-dimensional degenerate SSC problems. In particular \cite{GuoTom09} looks at bounded variation control problems of \emph{maximisation} for objective functionals which are \emph{concave} with respect to the control variable and one of their results states that the smooth fit certainly holds if the running profit (i.e.~their counterpart of our function $x\Phi(c)$) lies in $C^2$. It is therefore interesting to observe that in the present paper we indeed have a smooth running cost and the break down of smooth fit is a consequence exclusively of the lack of convexity (in $\nu$) of the cost functional. To the best of our knowledge, this phenomenon is a novelty in the literature.

The rest of the paper is organised as follows. In Section \ref{Setting} we set up the problem and make some standing assumptions. In Section \ref{sec:heuristic} we provide a heuristic study of the action region and of the optimal control, and then we state the main results of the paper (see Theorems \ref{thm:main-nu} and \ref{thm:main-U}) which provide a full solution to \eqref{problem-intro}. Section \ref{construction} is devoted to proving all the technical steps needed to obtain the main result and it follows a constructive approach validated at the end by a verification argument. Finally, proofs of some results needed in Section \ref{construction} are collected in Appendix \ref{someproofs}.


\section{Setting and assumptions}
\label{Setting}

Let $(\Omega,\cF,\PP)$ be a complete probability space carrying a one-dimensional standard Brownian motion $(B_t)_{t\ge0}$ adapted to its natural filtration augmented by $\PP$-null sets $\mathbb{F}:=(\cF_t)_{t\ge0}$. We denote by $X^x$ the Brownian motion starting from $x \in \RR$ at time zero,
\beq
\label{statevariable}
X^x_t= x + B_t, \quad t\geq 0.
\eeq
It is well known that $X^x$ is a {(null)} recurrent process with infinitesimal generator $\mathbb{L}_X:= \frac{1}{2}\frac{d^2}{dx^2}$ and with fundamental decreasing and increasing solutions of the characteristic equation $(\mathbb{L}_X - \lambda)u =0$ given by $\phi_{\lambda}(x):=e^{-\sqrt{2\lambda}x}$ and $\psi_{\lambda}(x):=e^{\sqrt{2\lambda}x}$, respectively.

Letting $c \in [0,1]$ be constant, we denote by $C^{c,\nu}$ the purely controlled process evolving according to
\beq
\label{ControlledY}
C^{c,\nu}_t= c + \nu_t, \quad t \geq 0,
\eeq
where $\nu$ is a control process belonging to the set
\begin{eqnarray}
\label{admissiblecontrols}
\mathcal{A}_c \hspace{-0.2cm}&:=&\hspace{-0.2cm} \{\nu:\Omega \times \mathbb{R}_{+} \mapsto  \mathbb{R}_{+}, ({\nu_{t}(\omega) := \nu(\omega,t)})_{t \geq 0}\mbox{ is nondecreasing,\,\,left-continuous, adapted} \nonumber \\
&& \hspace{5cm} \mbox{ with} \,\,c + \nu_t\leq 1\,\,\,\forall \; t \geq 0, \nonumber \,\,\nu_0=0\,\,\,\,\PP-\mbox{a.s.}\}. \nonumber
\end{eqnarray}
From now on controls belonging to $\mathcal{A}_c$ will be called {{\em admissible}}.

Given a positive discount factor $\lambda$ and a
running cost function $\Phi$, {our problem \eqref{problem-intro}} is to find
\begin{align}
\label{valuefunction}
U(x,c):=\inf_{\nu \in \mathcal{A}_c} \mathcal{J}_{x,c}(\nu),
\end{align}
with
\begin{equation}
\label{nonconvex}
\mathcal{J}_{x,c}(\nu):=\EE\bigg[\int_0^{\infty}e^{-\lambda s}\lambda X^x_{s}\Phi(
C^{c,\nu}_{s})ds + \int^{\infty}_0 {e^{-\lambda s}X^x_s\,d{\nu}_s} \bigg],
\end{equation}
and to determine a minimising control policy $\nu^*$ if one exists.
A priori existence results for the optimal solutions of SSC problems with cost criteria which are not necessarily convex are rare in the literature. Two papers dealing with questions of such existence in abstract form are \cite{DufourMiller} and \cite{HaussmanSuo}. Here we do not provide any abstract existence result for the optimal policy of problem \eqref{valuefunction}, but we explicitly construct it in Section \ref{construction} below.

Throughout this paper, for $t \geq 0$ and $\nu \in \mathcal{A}_c$ we will make use of the notation $\int_0^{t} e^{-\lambda s}X^x_s d\nu_s$ to indicate the Stieltjes integral $\int_{[0,t)} e^{-\lambda s} X^x_s d\nu_s$ with respect to $\nu$. Moreover, from now on the following standing assumption on the running cost factor $\Phi$ will hold.
\begin{ass}
\label{ass-Phi}
$\Phi: \mathbb{R} \mapsto \mathbb{R}_+$ lies in $C^2(\mathbb{R})$ and is decreasing and strictly convex with $\Phi(1)=0$.
\end{ass}

We will observe in Section \ref{sec:heuristic} below that the sign of 
\beq
\label{def-k}
k(c):=\lambda+\lambda \,\Phi'(c)
\eeq
plays a crucial role. We now define also the function
\begin{align}
\label{def:R}
R(c):=1-c-\Phi(c),\qquad c\in[0,1],
\end{align}
and assume the existence of constants $\hat c$ and $c_o$, both lying in $(0,1)$, such that
\begin{eqnarray}
\label{def-co}
R(c_o)&=&0, \text{ and } \\
\label{def-chat}
R'(\hat{c})&=&0 \text{ (or equivalently, }  k(\hat{c})=0).
\end{eqnarray}
It follows from the strict convexity of $\Phi$ that the function $c \mapsto k(c)$ is strictly increasing and that $\hat c$, $c_o$ are uniquely defined. The assumption that $c_o$ lies in $(0,1)$ allows us to consider the most general setting but the case where $c_o$ does not exist in $(0,1)$ is also covered by the method presented in the next sections. The next result easily follows from properties of $\Phi$.
\begin{lemma}
\label{lem:R}
$R(1)=0$ and $R$ is strictly concave, hence it is negative on $[0,c_o)$ and positive on $(c_o,1)$; also, $R$ has a positive maximum at $\hat{c}$ and therefore $c_o < \hat{c}$.
\end{lemma}


\section{Preliminary discussion and main results}
\label{sec:heuristic}

In order to derive a candidate solution to problem \eqref{valuefunction} we perform a preliminary heuristic analysis, distinguishing three cases according to the signs of $k(c)$ and $x$.

(A). When $k(c)>0$ (i.e.~when $c\in(\hat{c},1]$) consider the costs of the following three strategies exerting respectively: no control, a small amount of control, or a large amount. Firstly if control is never exercised, i.e.~$\nu_t\equiv0$, $t\ge0$, one obtains from \eqref{nonconvex} an overall cost $\mathcal{J}_{x,c}(0)=x\Phi(c)$ by an application of Fubini's theorem. If instead at time zero one increases the inventory by a small amount $\delta>0$ and then does nothing for the remaining time, i.e.~$\nu_t=\nu^\delta_t:=\delta$ for $t>0$ in \eqref{nonconvex}, the total cost is $\mathcal{J}_{x,c}(\nu^\delta)=x(\delta+\Phi(c+\delta))$. {Writing} $\Phi(c+\delta)=\Phi(c)+\Phi'(c)\delta+o(\delta^2)$ we find that $\mathcal{J}_{x,c}(\nu^\delta)=\mathcal{J}_{x,c}(0)+\delta x(1+\Phi'(c))+o(\delta^2)$ so that exercising a small amount of control reduces future {expected} costs relative to a complete inaction strategy only if $x k(c)/\lambda<0$, i.e.~$x<0$, since $k(c)>0$. It is then natural to expect that for each $c\in(\hat{c},1]$ there should exist $\gamma(c)<0$ such that it is optimal to exercise control only when $X^x_s \leq\gamma(c)$. 

We next want to understand whether a small control increment is more favourable than a large one and for this we consider a strategy where at time zero one exercises all available control, i.e.~$\nu_t=\nu^f_t:=1-c$ for $t>0$. The latter produces a total expected cost equal to $\mathcal{J}_{x,c}(\nu^f)=x(1-c)$, so that for $x<0$ and recalling that $k$ is increasing one has
\begin{align}\label{pre00}
\mathcal{J}_{x,c}(\nu^f)-\mathcal{J}_{x,c}(\nu^\delta)=\frac{x}{\lambda}\Big(\int_c^1k(y)dy - \delta k(c)\Big)+o(\delta^2)\le \frac{x}{\lambda}k(c)(1-c-\delta).
\end{align}
Since $k(c)>0$ the last expression is negative whenever $1-c>\delta$, so it is reasonable to expect that large control increments are more profitable than small ones. This suggests that the threshold $\gamma$ introduced above should not be of the reflecting type (see for instance \cite{MehriZervos}) but rather of repelling type as observed in \cite{DZ} and \cite{KOWZ} among others. Using this heuristic a corresponding free boundary problem is formulated and solved in Section \ref{sec:step1}.

(B1). When $k(c)<0$ (that is, when $c\in[0,\hat{c})$) we again compare inaction to small and large control increments. Observe that now $\nu^\delta$ is {favourable} (with respect to complete inaction) if and only if $x k(c)/\lambda<0$, i.e.~$x>0$, since now $k(c)<0$. Hence we expect that for fixed $c\in[0,\hat{c})$ one should act when the process $X$ exceeds a positive upper threshold $\beta(c)$. Then compare a small control increment with a large one, in particular consider a policy $\nu^{\hat{c}}$ that immediately exercises an amount $\hat{c}-c$ of control and then acts optimally for problem \eqref{valuefunction} with initial conditions $(x,\hat{c})$. The expected cost associated to $\nu^{\hat{c}}$ is $\mathcal{J}_{x,c}(\nu^{\hat{c}})=x(\hat{c}-c)+U(x,\hat{c})$ and one has
\begin{align}\label{pre01}
\mathcal{J}_{x,c}(\nu^{\hat{c}})-\mathcal{J}_{x,c}(\nu^\delta)\le\frac{x}{\lambda}\Big(\int_c^{\hat{c}}k(y)dy - \delta k(c)\Big)+o(\delta^2)
\end{align}
where we have used that $U(x,\hat{c})\le x\Phi(\hat{c})$. If we fix $c\in[0,\hat{c})$ and $x>0$, then for $\delta>0$ sufficiently small the right-hand side of \eqref{pre01} becomes negative, which suggests that a reflection strategy at the upper boundary $\beta$ would be less {favourable}
than the strategy described by $\nu^{\hat{c}}$.

(B2). Finally, when $x<0$ and $k(c)<0$ we compare the `large' increment to inaction. Note that $U(x,\hat{c})\le x(1-\hat{c})$ to obtain
\begin{align}\label{pre02}
\mathcal{J}_{x,c}(\nu^{\hat{c}})-\mathcal{J}_{x,c}(0)\le\frac{x}{\lambda}\int_c^{1}k(y)dy=\frac{x}{\lambda}\Big(\int_c^{\hat{c}}k(y)dy
+\int_{\hat{c}}^1k(y)dy\Big).
\end{align}
The first integral on the right-hand side of \eqref{pre02} is negative but its absolute value can be made arbitrarily small by taking $c$ close to $\hat{c}$. The second integral is positive and {independent} of $c$. Thus the overall expression becomes negative when $c$ approaches $\hat{c}$ from the left. This suggests that when the inventory is {a little below} the critical value $\hat{c}$ an investment sufficient to increase the inventory to the level $\hat{c}$ is preferable to inaction, {after which the optimisation continues as discussed above for $c \in (\hat c, 1]$. We therefore explore the presence of both upper and lower repelling boundaries when $c$ is a little below $\hat{c}$, and this is done in Section \ref{sec:step2}. The candidate solution for smaller values of $c$ (when the heuristic suggests only an upper boundary) is constructed in Section \ref{clessco}.
\vspace{+7pt}

In each of the previous heuristics it is preferable to exert a large amount of control. This suggests suitable connections to optimal stopping problems (although not necessarily of differential type) and the main novelty in this paper is to exploit these expected connections. In particular we take advantage of the opportunity to solve optimal stopping problems using geometric arguments as in \cite{DayKar}. This allows candidates for the control boundaries, value function and optimal control policy to be constructed and analytical properties to be derived. Further these optimal stopping problems have similar variational inequalities to the control problem, which facilitates verification of the candidate solution.

Before proceeding with the formal analysis we present the solution to problem \eqref{valuefunction}, which is our main result. As suggested by the above heuristics the solution is somewhat complex, but a straightforward graphical presentation is given in Figure \ref{fig:1}. The formal solution is given in the next three results.

\begin{theorem}
\label{thm:main-nu}
Recall $\hat{c}$ and $c_o$ as in \eqref{def-chat} and \eqref{def-co}, respectively. There exists two functions $\hat{\beta},\,\hat{\gamma}$ defined on $[0,1]$ and taking values in the extended real line $\RR\cup\{\pm\infty\}$ fulfilling 
\begin{itemize}
\item[ i)] In $[0,\hat{c})$, $\hat{\beta}\in(0,1/\sqrt{2\lambda})$, it is $C^1$ and decreasing, whereas for $c\in[\hat{c},1]$, $\hat{\beta}(c)=+\infty$; 
\item[ii)] In $(c_o,1]$, $\hat{\gamma}\le -1/\sqrt{2\lambda}$, it is $C^1$ and non decreasing, whereas for $c\in[0,c_o]$, $\hat{\gamma}(c)=-\infty$;
\end{itemize}
and such that an optimal control $\nu^*$ can be constructed as follows: 
for $(x,c)\in\RR\times(0,1)$ define the stopping times  
\begin{equation}
\label{taubetataugamma}
\tau_{\hat{\beta}}:=\inf\{t \geq 0: X^x_t \geq \hat{\beta}(c)\}, \qquad \tau_{\hat{\gamma}}:=\inf\{t \geq 0: X^x_t \leq \hat{\gamma}(c)\},
\end{equation}
and
\begin{equation}
\label{tausigmastar}
\tau^*:=\tau_{\hat{\beta}} \wedge \tau_{\hat{\gamma}}, \qquad \sigma^*:=\inf\{t \geq \tau_{\hat{\beta}}: X^x_t \leq \hat{\gamma}(\hat{c})\},
\end{equation}
with the convention $\inf\emptyset=+\infty$ (note that $\tau_{\hat{\beta}}=+\infty=\sigma^*$, $\PP$-a.s.~if $c\ge \hat{c}$); then the admissible purely discontinuous control
\begin{equation}
\label{op-contr01}
\nu^*_{t}:= (1-c)\mathds{1}_{\{t\,>\,\tau^*\}}\mathds{1}_{\{\tau^*\,=\,\tau_{\hat{\gamma}}\}} + \big[(\hat{c}-c)\mathds{1}_{\{t\, \leq\, \sigma^*\}} + (1-\hat{c})\mathds{1}_{\{t\, >  \,\sigma^*\}}\big]\mathds{1}_{\{t\,>\,\tau^*\}}\mathds{1}_{\{\tau^*\,=\,\tau_{\hat{\beta}}\}}
\end{equation}
is optimal for \eqref{valuefunction}.
\end{theorem}

\begin{prop}
\label{thm:main-bd}
The optimal boundaries $\hat{\beta}$ and $\hat{\gamma}$ of Theorem \ref{thm:main-nu} are characterised as follows:
\begin{itemize}
\item[i)] For $c\in[\hat{c},1]$ one has $\hat{\gamma}(c)=-1/\sqrt{2\lambda}$ (and $\hat{\beta}(c)=+\infty$); 
\item[ ii)] For $c\in(c_o,\hat{c})$ one has $\hat{\gamma}(c)=\tfrac{1}{2\sqrt{2\lambda}}\ln(\hat{y}_1(c))$ and $\hat{\beta}(c)=\tfrac{1}{2\sqrt{2\lambda}}\ln(\hat{y}_2(c))$ where $\hat{y}_1$ and $\hat{y}_2$ are the unique couple solving the following problem:
\begin{align}
&\text{Find $y_1\in(0,e^{-2})$ and $y_2\in(1,e^{2})$ such that $F_1(y_1,y_2;c)=0$ and $F_2(y_1,y_2;c)=0$ with}\nonumber\\[+3pt]
\label{F1} &F_1(x,y;c) :=x^{-\frac{1}{2}}(1 + \tfrac{1}{2}\ln x)R(c) - y^{-\frac{1}{2}}(1 + \tfrac{1}{2}\ln y)(R(c)-R(\hat{c})), \\[+3pt]
\label{F2} &F_2(x,y;c) :=x^{\frac{1}{2}}(1 - \tfrac{1}{2}\ln x)R(c) - y^{\frac{1}{2}}(1 - \tfrac{1}{2}\ln y)(R(c)-R(\hat{c})) - 2e^{-1}R(\hat{c});
\end{align}
\item[ iii)] For $c\in[0,c_o]$ one has $\hat{\beta}(c)=\tfrac{1}{2\sqrt{2\lambda}}\ln(\hat{y}_2(c))$ (and $\hat{\gamma}(c)=-\infty$) where $\hat{y}_2$ is the unique solution in $(1,e^{2})$ of $F_3(y;c)=0$ with
\begin{equation}
\label{F3}
F_3(y;c):=y^{\frac{1}{2}}\big(1 - \tfrac{1}{2}\ln y\big) - \frac{2e^{-1}R(\hat{c})}{R(\hat{c}) -R(c)}.
\end{equation} 
\end{itemize} 
\end{prop}

\begin{theorem}
\label{thm:main-U}
Let $\mathcal{O}:=\RR\times(0,1)$. The function $U$ of \eqref{valuefunction} belongs to $C^1(\mathcal{O})\cap C(\overline{\mathcal{O}})$ with $U_{xx}\in L^{\infty}_{loc}(\mathcal{O})$ and it solves the variational problem
\begin{align}\label{eq:HJB}
\max\big\{(-\tfrac{1}{2}w_{xx}+\lambda w)(x,c)-\lambda x \Phi(c)\,,\,-w_c(x,c)-x\big\}=0,\quad\text{for a.e.~$(x,c)\in\mathcal{O}$}
\end{align}
with $U(x,1)=0$, $x\in\RR$.
\end{theorem}

The boundaries $\hat{\beta}$ and $\hat{\gamma}$ of Proposition \ref{thm:main-bd} fully characterise the optimal control $\nu^*$ illustrated in Figure \ref{fig:1}, which prescribes to do nothing until the uncontrolled process $X^x$ leaves the interval $(\hat{\gamma}(c),\hat{\beta}(c))$, where $c\in[0,1)$ is the initial {inventory level}. Then, if $\tau_{\hat{\gamma}}<\tau_{\hat{\beta}}$ one should immediately exert all the available control after hitting the lower moving boundary $\hat{\gamma}(c)$. If instead $\tau_{\hat{\gamma}}>\tau_{\hat{\beta}}$ one should initially {increase the inventory to $\hat{c}$} after hitting the upper moving boundary $\hat{\beta}(c)$ and then wait until $X$ hits the new value $\hat{\gamma}(\hat{c})$ of the lower boundary before exerting all remaining available control. 

\begin{figure}[!ht]
\centering
\includegraphics[scale=0.5]{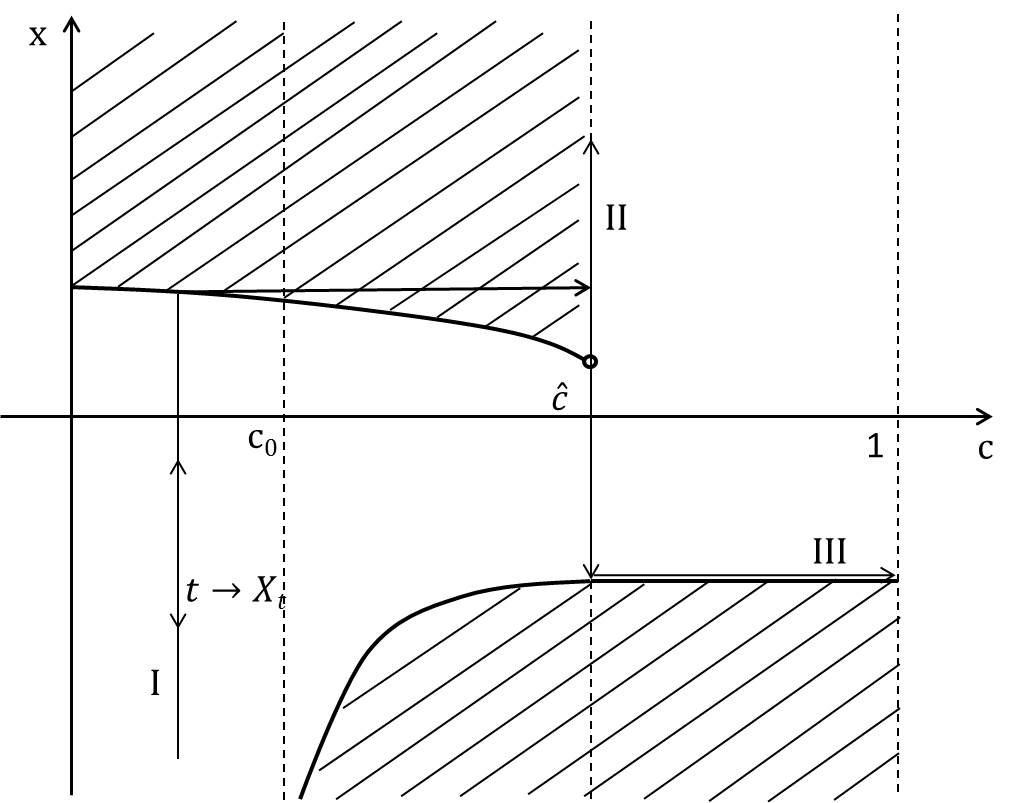}
\caption{\small An illustrative diagram of the the optimal boundaries and of the optimal control $\nu^*$ of \eqref{op-contr01}. The upper boundary $\hat{\beta}$ and the lower boundary $\hat{\gamma}$ split the state space into the inaction region (white) and action region (hatched). When the initial state is $(x,c)$ with $c\in[0,c_o)$ and $x<\hat{\beta}(c)$ one observes the following three regimes: in regime $(I)$ the process $X$ diffuses until hitting $\hat{\beta}(c)$, then {an amount $\Delta\nu=\hat{c}-c$ of control is exerted}, horizontally pushing  the process $(X,C)$ to the regime $(II)$; there $X$ continues to diffuse until it hits $\gamma^o$ and at that point all {remaining control is exercised} and $(X,C)$ is pushed horizontally until the inventory reaches its maximum $(III)$.}\label{fig:1}
\end{figure}

\section{Construction of a candidate value function}
\label{construction}

The direct solution of \eqref{eq:HJB} is challenging in general as it is a free boundary problem with (multiple) non constant boundaries. When necessary, however, for each fixed value of $c$ we will identify an associated optimal stopping problem whose solution is simpler since its free boundaries are given by two points. Our candidate solution $W$ is then effectively obtained by piecing together partial solutions on different domains. 
More precisely, recalling the definitions of $c_o$ and $\hat c$ from \eqref{def-co}-\eqref{def-chat}, we carry out the following steps:
\begin{enumerate}
\item[Step 1)] Directly solve \eqref{eq:HJB} when the initial value of the inventory $c \in [\hat{c},1]$, obtaining a partial candidate solution $W^o$ (Section \ref{sec:step1}).
\item[Step 2)] Identify an associated (parameter-dependent) problem of optimal stopping for $c \in [0,\hat c)$ (Section \ref{sec:auxiliary}),
\item[Step 3)] Solve the stopping problems when $c \in (c_o,\hat{c})$ (Section \ref{sec:step2}; cf. heuristic (B2)).
\item[Step 4)] Solve the stopping problems when $c \in [0,c_o)$ (Section \ref{clessco}; cf. heuristic (B1)).
\item[Step 5)] Construct a partial candidate solution $W^1$ from the OS solutions of steps 3 and 4, verifying that it solves the variational problem \eqref{eq:HJB} when the initial value of the inventory $c \in [0,\hat{c})$ (Section \ref{sec:W1}),
\item[Step 6)] Paste together the partial candidate solutions from steps 1 and 5 to obtain the complete candidate solution $W$, verifying that it solves the variational problem \eqref{eq:HJB} when the initial value of the inventory $c \in [0,1]$ (Section \ref{verification}). 
\end{enumerate}

We begin the construction of a candidate value function by establishing the finiteness of the expression \eqref{valuefunction} under our assumptions. 
\begin{prop}
\label{sublineargrowth}
Let $U$ be as in \eqref{valuefunction}. Then there exists $K>0$ such that $|U(x,c)| \leq K(1 + |x|)$ for any $(x,c)\in \RR \times [0,1]$.
\end{prop}
\begin{proof}
Take $\nu \in \mathcal{A}_c$ and integrate by parts the cost term $\int_0^{\infty}e^{-\lambda s}X^x_s d\nu_s$ in \eqref{nonconvex} noting that $M_t:=\int_0^t e^{-\lambda s}\nu_s dB_s$ {is a uniformly integrable martingale}. Then by well known estimates for Brownian motion we obtain
\begin{equation}
\label{Usublinear}
|\mathcal{J}_{x,c}(\nu)|\leq \EE\bigg[\int_0^{\infty}e^{-\lambda s}\lambda |X^x_{s}|\big[\Phi(
C^{c,\nu}_{s}) + \nu_s\big]ds \bigg] \leq K(1 + |x|),
\end{equation}
for some $K>0$, since $\Phi(c) \leq \Phi(0)$, $c \in [0,1]$ by Assumption \ref{ass-Phi} and  $\nu\in \mathcal{A}_c$ is bounded from above by $1$. By \eqref{Usublinear} and the arbitrariness of $\nu \in \mathcal{A}_c$ the proposition is proved.
\end{proof}

\subsection{Step 1: initial value of inventory $c \in [\hat{c},1]$}
\label{sec:step1}
{We formulate the first heuristic of Section \ref{sec:heuristic} mathematically by} writing \eqref{eq:HJB} as a free boundary problem, to find the couple of functions $(u,\gamma)$, with $u\in C^1(\RR\times[\hat{c},1])$ and $U_{xx}\in L^{\infty}_{loc}(\RR\times(\hat{c},1))$, solving
\begin{align}\label{fbp-cbiggerhatc}
\left\{
\begin{array}{ll}
\tfrac{1}{2}u_{xx}(x,c)-\lambda u(x,c)=-\lambda x\Phi(c) & \text{for $x>\gamma(c)$, $c\in[\hat{c},1)$}\\[+4pt]
\tfrac{1}{2}u_{xx}(x,c)-\lambda u(x,c)\ge-\lambda x\Phi(c) & \text{for a.e.~$(x,c)\in\RR\times[\hat{c},1)$}\\[+4pt]
u_c(x,c)\ge -x & \text{for $x\in\RR$, $c\in[\hat{c},1)$}\\[+4pt]
u(x,c)=x(1-c)& \text{for $x\le \gamma(c)$, $c\in[\hat{c},1]$}\\[+4pt]
u_x(x,c)=(1-c)& \text{for $x\le \gamma(c)$, $c\in[\hat{c},1)$}\\[+4pt]
u(x,1)=0 &\text{for $x\in\RR$}.
\end{array}
\right.
\end{align}
\begin{prop}
\label{prop:solWo}
Recall $R$ from \eqref{def:R}. Then the couple $(W^o,\gamma^o)$ defined by $\gamma^o:=-\tfrac{1}{\sqrt{2\lambda}}$ and
\begin{equation}
\label{Wo}
W^o(x,c):=
\left\{
\begin{array}{ll}
-\tfrac{1}{\sqrt{2\lambda}}e^{-1}R(c)\phi_{\lambda} (x) + x \Phi(c), &  x > \gamma^o,\\[+6pt]
x(1-c), & x \leq \gamma^o,
\end{array}
\right.
\end{equation}
solves \eqref{fbp-cbiggerhatc} with $W^o\in C^1(\RR\times[\hat{c},1])$ and $W^o_{xx}\in L^{\infty}_{loc}(\RR\times(\hat{c},1))$.
\end{prop}
\begin{proof}
A general solution to the first equation in \eqref{fbp-cbiggerhatc} is given by $$u(x,c)= A^o(c)\psi_{\lambda}(x) + B^o(c)\phi_{\lambda}(x) + x\Phi(c), \qquad x >\gamma(c),$$ with $A^o$, $B^o$ and $\gamma$ to be determined.
Since $\psi_{\lambda}(x)$ diverges with a superlinear trend as $x\to\infty$ and $U$ has sublinear growth by Proposition \ref{sublineargrowth}, we set $A^o(c)\equiv 0$.
Imposing the fourth and fifth {conditions} of \eqref{fbp-cbiggerhatc} for $x=\gamma(c)$ and recalling the expression for $R$  in \eqref{def:R} we have
\begin{equation}
\label{Aogammao}
B^o(c):=-\tfrac{1}{\sqrt{2\lambda}}e^{-1}R(c), \qquad \gamma(c)=\gamma^o=-\tfrac{1}{\sqrt{2\lambda}}.
\end{equation}
This way the function $W^o$ of \eqref{Wo} clearly satisfies $W^o(x,1)=0$, $W^o_x$ is continuous by construction and by some algebra it is not difficult to see that $W^o_c$ is continuous on $\RR\times[\hat{c},1]$ with $W^o_c(\gamma^o,c) = -\gamma^o$, $c\in[\hat{c},1]$. Moreover one also has
\begin{align}\label{eq:BM-smfitWo}
W^o_{cx}(x,c)+1=(1+\Phi'(c))\big(1-e^{-1}\phi_\lambda(x)\big)\ge0,\qquad x>\gamma^o, c\in[\hat{c},1],
\end{align}
and hence $W^o_{cx}(\gamma^o+,c)=-1$, for $c\in[\hat{c},1]$, i.e.~the smooth fit condition holds, and $W^o_c(x,c)\ge-x$ on $\RR\times[\hat{c},1]$ as required. It should be noted that $W^o_{xx}$ fails to be continuous across the boundary although it remains bounded on any compact subset of $\RR\times[\hat{c},1]$.

Finally we observe that
\begin{align}\label{eq:BM-subharm}
\tfrac{1}{2}W^o_{xx}(x,c)-\lambda W^o(x,c)=-\lambda x(1-c)\ge-\lambda x \Phi(c)\quad\text{for $x\le\gamma^o$, $c\in[\hat{c},1]$,}
\end{align}
since $\gamma^o<0$ and $R(c)\ge0$ on $c\in[\hat{c},1]$.
\end{proof}

\begin{remark}
\label{rem:connection} 
We may observe a double connection to optimal stopping problems here, as follows:

1. We could have applied heuristic (A) from Section \ref{sec:heuristic}, approaching this sub-problem as one of optimal stopping. However the free boundary turns out to be constant for $c \in [\hat c,1]$ and the direct solution of \eqref{eq:HJB} is straightforward in this case. Links to OS are, however, more convenient in the following sections.

2. Alternatively we may differentiate the explicit solution \eqref{Wo} with respect to $c$. Then holding $c\in[\hat{c},1)$ constant it is straightforward to confirm that $W^o_c$ solves the free boundary problem associated to the following OS problem:
\begin{align}\label{eq:OSw}
w(x,c):=\sup_{\tau\ge0}\EE\Big[\lambda\Phi'(c)\int_0^\tau{e^{-\lambda t} X^x_t dt}-e^{-\lambda\tau}X^x_\tau\Big].
\end{align}
This differential connection to optimal stopping is formally the same as the differential connection previously observed in convex SSC problems (see~\cite{KaratzasShreve84}). 
\end{remark}


\subsection{Step 2: an auxiliary problem of optimal stopping for $c \in [0,\hat c)$}
\label{sec:auxiliary}

We now use heuristics (B1) and (B2) from Section \ref{sec:heuristic} to identify an associated parametric family of optimal stopping problems, which are solved in this section.
More precisely we conjecture here (and will verify in Section \ref{verification}) that for $x\in\RR$ and $c\in[0,\hat{c})$, the value function $U(x,c)$ equals
\begin{equation}
\label{Wstar}
W^1(x,c):=\inf_{\tau \geq 0} \EE\bigg[\int_0^{\tau}e^{-\lambda t}\lambda X^x_t\Phi(c) dt + e^{-\lambda\tau}X^x_{\tau}(\hat{c} - c) + e^{-\lambda\tau}W^o(X^x_{\tau},\hat{c})\bigg],
\end{equation}
where the optimisation is taken over the set of $(\mathcal{F}_t)$-stopping times valued in $[0,\infty)$, $\PP$-a.s.
We begin by noting that It\^o's formula may be used to express \eqref{Wstar} as an OS problem in the form:
\begin{equation}
\label{Wstar2}
W^1(x,c)=x\Phi(c) + V(x,c),
\end{equation}
where
\begin{eqnarray}
\label{def-V}
V(x,c)&:=&\inf_{\tau \geq 0} \EE\big[ e^{-\lambda\tau}G(X^x_{\tau},c)\big], \\
G(x,c)&:=&x(\hat{c}- c -\Phi(c)) + W^o(x,\hat{c}). \label{def-G}
\end{eqnarray}
Note that $G\in C(\RR\times[0,\hat{c}])$, $|G(x,c)|\le C(1+|x|)$ for suitable $C>0$ and {$x\mapsto \EE\big[ e^{-\lambda\tau}G(X^x_{\tau},c)\big]$ is continuous for any fixed $\tau$ and $c\in[0,\hat{c})$}. Then from standard theory an optimal stopping time is $\tau_*:=\inf\{t\ge0:X^x_t\in\cS_c\}$ where
\begin{align}
\label{regions}
\cC_c: = \{x\in\RR: V(x,c) < G(x,c)\} \quad \mbox{and} \quad \cS_c := \{x\in\RR: V(x,c)=G(x,c)\}
\end{align}
are continuation and stopping regions respectively, and $V$ is finite valued. 

The solution of the parameter-dependent optimal stopping problem \eqref{def-V} is somewhat complex and we will apply the geometric approach originally introduced in \cite{Dynkin}, Chapter 3, for Brownian motion and expanded in \cite{DayKar}. The solutions are illustrated in Figures \ref{fig:H1} and \ref{fig:H2} in a sense which will be clarified in Proposition \ref{prop:DayKar}. This allows the analytical characterisation of the optimal stopping boundaries as $c$ varies and thus the study of their properties, avoiding the difficulties encountered in the more direct approach of \cite{DeAFeMo14}.
As in \cite{DayKar}, eq.\ (4.6), we define
\begin{equation}
\label{def-F}
F_{\lambda}(x):=\frac{\psi_{\lambda}(x)}{\phi_{\lambda}(x)} = e^{2\sqrt{2\lambda}x}, \qquad x \in \RR,
\end{equation}
together with its inverse
\begin{equation}
\label{def-Finverse}
F_{\lambda}^{-1}(y)=\frac{1}{2\sqrt{2\lambda}}\ln(y), \qquad y > 0,
\end{equation}
and the function
\begin{align}\label{def-H}
H(y,c):=
\left\{
\begin{array}{ll}
\frac{G(F^{-1}_\lambda(y),c)}{\phi_\lambda(F^{-1}_\lambda(y))}, & y>0\\[+4pt]
0 & y=0.
\end{array}
\right.
\end{align}
We can now restate part of Proposition 5.12 and Remark 5.13 of \cite{DayKar} as follows.
\begin{prop}\label{prop:DayKar}
Fix $c\in {[0,}\hat{c})$ and let $Q(\,\cdot\,,c)$ be the largest non-positive convex minorant of $H(\,\cdot\,,c)$ (cf.~\eqref{def-H}), then $V(x,c)=\phi_\lambda(x)Q(F_\lambda(x),c)$ for all $x\in\RR$. Moreover $\cS_c=F^{-1}_\lambda(\cS^Q_c)$, where $\cS^Q_c:=\{y>0:Q(y,c)=H(y,c)\}$ (cf.~\eqref{regions}).
\end{prop}
Note that characterising $W^1$ is equivalent to characterising $V$, which is in turn equivalent to finding $Q$. The latter and its contact sets $\cS^Q_c$ will be the object of our study in Sections \ref{sec:step2} and \ref{clessco}. 
Fixing $c\in[0,\hat{c})$, we first establish regularity properties of $H$. We have (from \eqref{Wo} and \eqref{def-G})
\begin{equation}
\label{def-G2}
G(x,c)=
\left\{
\begin{array}{ll}
xR(c), & x \le \gamma^o\\[+4pt]
-\frac{1}{\sqrt{2\lambda}}e^{-1}R(\hat{c})\phi_{\lambda}(x) + x(R(c)- R(\hat{c})), & x > \gamma^o.
\end{array}
\right.
\end{equation}
Noting that $\phi_{\lambda}(F_{\lambda}^{-1}(y))=y^{-\frac{1}{2}}$, $y > 0$, we obtain
\begin{equation}
\label{def-H2}
H(y,c)=
\left\{
\begin{array}{lr}
0, & y=0\\[+4pt]
\tfrac{1}{2\sqrt{2\lambda}}R(c) y^{\frac{1}{2}} \ln y, &  0 < y \leq e^{-2}\\[+4pt]
-\tfrac{1}{\sqrt{2\lambda}}e^{-1}R(\hat{c}) +\tfrac{1}{2\sqrt{2\lambda}}  (R(c) - R(\hat{c})) y^{\frac{1}{2}}\ln y, &  \qquad y > e^{-2}.
\end{array}
\right.
\end{equation}

\begin{lemma}\label{lem:Hreg}
The function $H$ belongs to $C^1((0,\infty)\times[0,\hat{c}])\cap C([0,\infty)\times[0,\hat{c}])$ with $H_{yy}\in L^\infty([\delta,\infty)\times[0,\hat{c}])$ for all $\delta>0$ and $H_{yc}\in C((0,\infty)\times[0,\hat{c}])$.
\end{lemma}
\begin{proof}
Since $G$ is continuous in $(x,c)$ the function $H$ is continuous on $(0,\infty)\times[0,\hat{c}]$ by construction and it is easy to verify that $\lim_{(y',c')\to(0,c)}H(y,c)=0$ for any $c\in[0,\hat{c}]$.
Since
\begin{align}
\label{def-Hy}
H_y(y,c)=\frac{1}{2\sqrt{2\lambda}}y^{-\frac{1}{2}} (1 + \frac{1}{2}\ln y)\times
\left\{
\begin{array}{lr}
R(c) , & 0 < y \le e^{-2}\\[+4pt]
R(c) - R(\hat{c}), &  y > e^{-2},
\end{array}
\right. 
\end{align}
then for any $c\in[0,\hat{c}]$, letting $(y_n,c_n)\to (e^{-2},c)$ as $n\to\infty$, $c_n \in [0,\hat{c})$, one has 
\begin{align*}
\lim_{n\to\infty,\: y_n<e^{-2}} H_y(y_n,c_n)=\lim_{n\to\infty,\: y_n>e^{-2}} H_y(y_n,c_n)=0,
\end{align*}
hence $H_y$ is continuous on $(0,\infty)\times[0,\hat{c}]$. Moreover we also have
\begin{align}
\label{Hc} H_c(y;c)&=\frac{1}{2\sqrt{2\lambda}} R'(c)y^{\frac{1}{2}}\ln y \quad \text{on $(0,\infty)\times[0,\hat{c}]$}\\
\label{Hyc} H_{yc}(y;c)&= R'(c) \frac{1}{2\sqrt{2\lambda}}y^{-\frac{1}{2}} (1 + \frac{1}{2}\ln y) \quad \text{on $(0,\infty)\times[0,\hat{c}]$}\\
\label{def-Hyy}
H_{yy}(y;c)&= - \frac{y^{-\frac{3}{2}}}{8\sqrt{2\lambda}}\ln(y) \times
\left\{
\begin{array}{lr}
R(c), & 0 < y \le e^{-2}\\[+4pt]
R(c) - R(\hat{c}), & \qquad y > e^{-2},
\end{array}
\right.
\end{align}
so that the remaining claims easily follow. 
\end{proof}

The sign of $R(c)$ (defined in \eqref{def:R}) will play an important role in determining the geometry of the obstacle $H$. Recalling that $c_o$ is the unique root of $R$ in $(0,1)$, in Sections \ref{sec:step2} and \ref{clessco}
we consider the cases $c \in [0,c_o)$ and $c \in (c_o,\hat c)$ respectively. 
The intermediate case $c=c_o$ is obtained by pasting together the former two in the limits as $c\uparrow c_0$ and $c\downarrow c_0$ and noting that these limits coincide.

\subsubsection{Step 3: initial value of inventory $c \in (c_o, \hat{c})$}
\label{sec:step2}

For $c \in (c_o, \hat{c})$, so that $R(c) > 0$ and $k(c)<0$, we now apply heuristic (B2). Lemma \ref{lem:H} collects some geometric properties of $H$ while Proposition \ref{existencegeometric} enables us to establish that in the present case, the minorant of Proposition \ref{prop:DayKar} has the form illustrated in Figure \ref{fig:H1}.

\begin{figure}[!ht]
\centering
\includegraphics[scale=0.5]{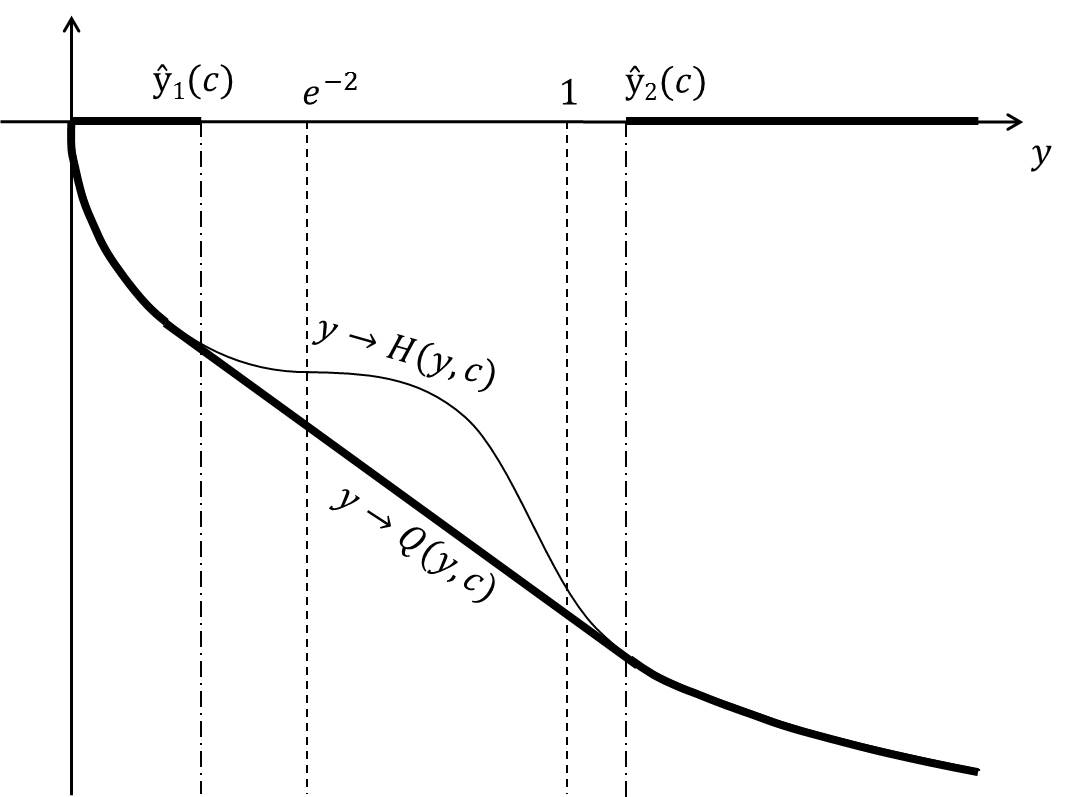}
\caption{\small An illustrative plot of the functions $y \mapsto H(y,c)$ and $y \mapsto Q(y,c)$ (bold) of \eqref{def-H2} and \eqref{WDayKar}, respectively, for fixed $c \in (c_o, \hat{c})$. The bold region $[0,\hat{y}_1(c)]\cup [\hat{y}_2(c),\infty)$ on the $y$-axis is the stopping region $\cS^Q_c$.}
\label{fig:H1}
\end{figure}

\begin{lemma}
\label{lem:H}
Let $c\in(c_o,\hat{c})$ be arbitrary but fixed. The function $H(\,\cdot\,,c)$ is strictly decreasing, with $\lim_{y \downarrow 0}H_y(y,c)=-\infty$ and $\lim_{y \uparrow \infty}H_y(y,c)=0$, and $H(\,\cdot\,,c)$ is strictly convex on $[0,e^{-2})\cup(1,\infty)$ and concave in $[e^{-2},1]$.
\end{lemma}
\begin{proof}
The proof is a simple consequence of \eqref{def-Hy}, \eqref{def-Hyy} and Lemma \ref{lem:R}, since we assume $c \in( c_o, \hat{c})$.
\end{proof}

The next proposition uses these properties to uniquely define a straight line tangent to $H(\cdot,c)$ at two points $\hat{y}_1(c)< e^{-2}$ and $\hat{y}_2(c)>1$, which will be used to define the moving boundaries $c \mapsto \hat{\beta}(c)$ and $c \mapsto \hat{\gamma}(c)$ introduced in Proposition \ref{thm:main-bd}. The convexity/concavity of $H(\cdot,c)$ then guarantee that the largest non-positive convex minorant $Q(\cdot,c)$ of Proposition \ref{prop:DayKar} is equal to this line on $(\hat{y}_1(c),\hat{y}_2(c))$ and equal to $H(\cdot,c)$ otherwise.
\begin{prop}
\label{existencegeometric}
For any $c\in (c_o,\hat{c})$ there exists a unique couple $(\hat{y}_1(c),\hat{y}_2(c))$ solving the system
\begin{equation}
\label{system-y1y2}
\left\{
\begin{array}{lr}
H_{y}(y_1,c)= H_y(y_2,c) \\[+4pt]
H(y_1,c) -H_y(y_1,c)y_1 = H(y_2,c) -H_y(y_2,c)y_2
\end{array}
\right.
\end{equation} 
with $\hat{y}_1(c)\in(0,e^{-2})$ and $\hat{y}_2(c)>1$.
\end{prop}
\begin{proof}
Define
\begin{align}\label{def-r}
r_y(z)&:=H_y(y,c)(z-y)+H(y,c), \qquad {y \geq 1},\,\,z\geq 0,
\\[+5pt]
g(y):=r_y(0) &= -\tfrac{1}{\sqrt{2\lambda}}e^{-1}R(\hat{c}) +\tfrac{1}{2\sqrt{2\lambda}}  (R(c) - R(\hat{c})) y^{\frac{1}{2}}\left(\frac 1 2 \ln y - 1 \right),
\\[+5pt]
P_r(y,c)&:=\sup_{z \in [0, 1]}h(z,y,c),\:\: \text{ where }\:\: h(z,y,c):=r_y(z)-H(z,c),\label{def-P_r}
\end{align}
so that $r_y(\,\cdot\,)$ is the straight line tangent to $H(\cdot,c)$ at $y$, with vertical intercept $g(y)$.
The function $y \mapsto P_r(y,c)$ is decreasing and continuous and it is clear that $P_r(1,c)>0$, since $H(\cdot,c)$ is concave on $[e^{-2},1]$. To establish the existence of a unique $\hat{y}_2(c)>1$ such that $P_r(\hat y_2(c),c)=0$, it is therefore sufficient to find $y > 1$ with $P_r(y,c)<0$.
Such a $y$ exists since $g(y) \rightarrow -\infty$ as $y \to \infty$:
it is clear from \eqref{def-P_r} that if $g(y)<H(1,c)$ then $P_r(y,c)<0$. 

Note that the map $z \mapsto h(z,y,c)$ is continuous, $h(1,y,c)<0$ for $y>1$ (cf. Figure \ref{fig:H1}) and $P_r(\hat y_2(c),c)=0$; then when $y=\hat y_2(c)$, the supremum in \eqref{def-P_r} is attained on the convex portion of $H(\cdot,c)$ (i.e.~in the interior of $[0,1]$) and thus is attained uniquely at a point $\hat y_1(c)\in(0,e^{-2})$.
By construction $(\hat{y}_1(c), \hat{y}_2(c))$ uniquely solves system \eqref{system-y1y2}.
\end{proof}

For $c\in (c_o,\hat{c})$ the minorant $Q$ is therefore
\begin{equation}
\label{WDayKar}
Q(y,c)=
\left\{
\begin{array}{ll}
H(y,c), & y \in [0, \hat{y}_1(c)],\\[+6pt]
H_y(\hat{y}_2(c),c)(y - \hat{y}_2(c)) + H(\hat{y}_2(c),c), & y \in (\hat{y}_1(c), \hat{y}_2(c)),\\[+6pt]
H(y,c), & y \in [\hat{y}_2(c),\infty).
\end{array}
\right.
\end{equation}

The following proposition
is proved in Appendix \ref{someproofs}.

\begin{prop}
\label{monboundaries}
The functions $\hat{y}_1$ and $\hat{y}_2$ of Proposition \ref{existencegeometric} belong to $C^1(c_o,\hat{c})$ with $c\mapsto \hat{y}_1(c)$ increasing and $c\mapsto \hat{y}_2(c)$ decreasing on $(c_o,\hat{c})$ and
\begin{enumerate}
	\item $\lim_{c \uparrow \hat{c}}\hat{y}_1(c) = e^{-2}$;
	\item $\lim_{c \downarrow c_o}\hat{y}_1(c) = 0$; \label{y1limco}
	\item $\hat{y}_2(c) < e^2$ for all $c\in(c_o,\hat{c})$;
	\item $\lim_{c \uparrow \hat{c}}\hat{y}^{'}_1(c) = 0$.
\end{enumerate}
\end{prop}


\subsubsection{{Step 4: initial value of inventory $c \in [0,c_o)$.}}
\label{clessco}

The geometry indicated in Figure \ref{fig:H1} does not hold in general. Indeed in heuristic (B2), a lower repelling boundary is suggested only for values of $c$ close to $\hat c$. It turns out that `close' in this sense means greater than or equal to $c_o$. We now take $c \in [0,c_o)$ and show that in this case the geometry of the auxiliary optimal stopping problems is as in Figure \ref{fig:H2}, so that each of these problems (which are parametrised by $c$) has just one boundary.
The following Lemma has a proof very similar to that of Lemma \ref{lem:H} and it is therefore omitted.

\begin{figure}[!ht]
\centering
\includegraphics[scale=0.35]{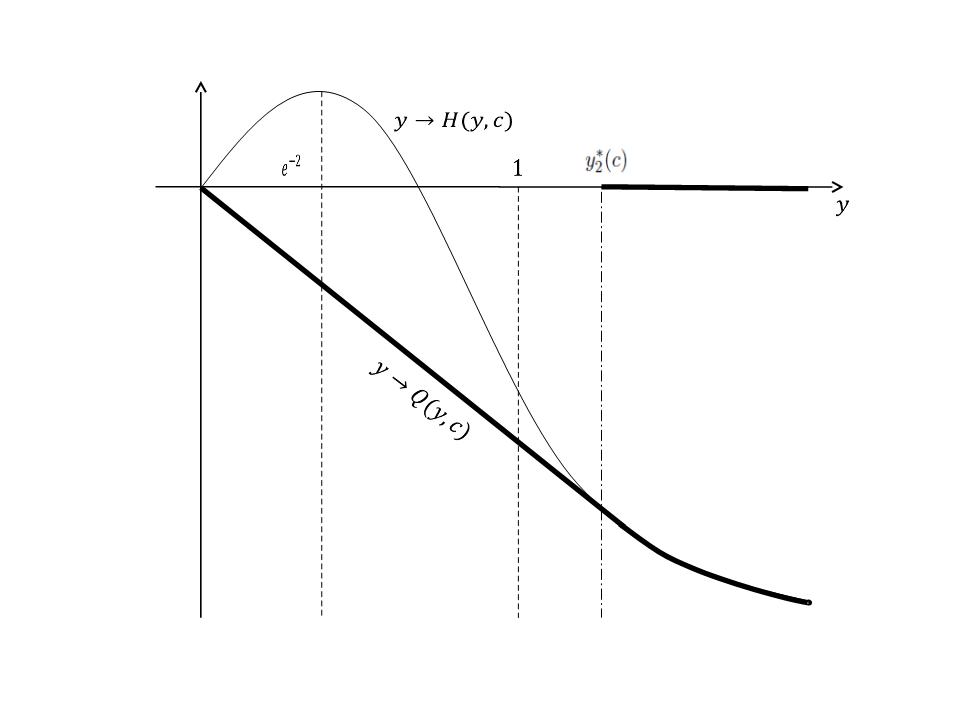}
\caption{\small An illustrative plot of the functions $y \mapsto H(y,c)$ and $y \mapsto Q(y,c)$ (bold) of \eqref{def-H2} and \eqref{WDayKar-clessc0}, respectively, for fixed $c \in [0,c_o)$. The bold interval $[y^*_2(c),\infty)$ on the $y$-axis is the stopping region $\cS^Q_c$.}
\label{fig:H2}
\end{figure}

\begin{lemma}
\label{lem:H2}
Let $c\in[0,c_o)$ be arbitrary but fixed. The function $H(\,\cdot\,,c)$ of \eqref{def-H2} is strictly increasing in $(0,e^{-2})$ and strictly decreasing in $(e^{-2},\infty)$. Moreover, $H(\,\cdot\,,c)$ is strictly concave in the interval $(0,1)$ and it is strictly convex in $(1,\infty)$ with $H_{yy}(1,c)=0$. 
\end{lemma}

The strict concavity of $H$ in $(0,1)$ suggests that there should exist a unique point $y_2^*(c) > 1$ solving
\begin{equation}
\label{defy2star}
H_y(y,c)y = H(y,c).
\end{equation}
The straight line $r_{y_2^*}:[0,\infty) \mapsto (-\infty,0]$ given by
$$r_{y_2^*}(y) := H(y_2^*(c),c) + H_y(y_2^*(c),c)(y -y_2^*(c))$$
is then tangent to $H$ at $y_2^*(c)$ and $r_{y_2^*}(0)=0$. 

The proof of the next result may be found in Appendix \ref{someproofs}.
\begin{prop}
\label{existencey2star}
For each $c\in[0,c_o)$ there exists a unique point $y_2^*(c) \in (1,e^2)$ solving \eqref{defy2star}. The function $c\mapsto y_2^*(c)$ is decreasing and belongs to $C^1([0,c_o))$. Moreover, for $\hat{y}_2$ as in Proposition \ref{existencegeometric} one has
\begin{equation}
\label{limdxsx1}
y_2^*(c_o-):=\lim_{c \uparrow c_o}y_2^*(c) = \lim_{c \downarrow c_o}\hat{y}_2(c)=:\hat{y}_2(c_o+)
\end{equation}
and
\begin{equation}
\label{limdxsx2}
(y_2^*)'(c_o-):=\lim_{c \uparrow c_o}(y_2^*)'(c) = \lim_{c \downarrow c_o}(\hat{y}_2)'(c)=:(\hat{y}_2)'(c_o+).
\end{equation}
\end{prop}

For $c\in [0,c_o)$ the minorant $Q$ is therefore
\begin{equation}
\label{WDayKar-clessc0}
Q(y,c)=
\left\{
\begin{array}{ll}
H_y(y_2^*(c),c)y , & y \in [0, y_2^*(c)),\\[+6pt]
H(y,c), & y \in [y_2^*(c),\infty).
\end{array}
\right.
\end{equation}
{Note that \eqref{WDayKar-clessc0} may be rewritten in the form \eqref{WDayKar} taking $\hat{y}_1(c)=0$ and replacing $\hat{y}_2$ by $y_2^*$, since $y^*_2$ solves \eqref{defy2star}.}


\subsubsection{Step 5: partial candidate value function $W^1$}
\label{sec:W1}

In this section we paste together the solutions obtained in sections \ref{sec:step2} and \ref{clessco} across $c=c_o$, and then apply the transformation of Proposition \ref{prop:DayKar} to obtain $V(x,c) =  x\Phi(c) - W^1(x,c)$ (recall \eqref{Wstar2}) and thus the partial candidate solution $W^1$ conjectured at the beginning of Section \ref{sec:auxiliary} for initial inventory levels $c \in [0,\hat c)$. We also establish a free boundary problem solved by $V$, which will help to show that $W^1$ solves \eqref{eq:HJB} for $c \in [0,\hat c)$.

The function $\hat{y}_1 \in C^1(c_o,\hat c)$ of Section \ref{sec:step2} may be extended to a function $\hat{y}_1 \in C^0([0,\hat c))$
by setting $\hat{y}_1(c)=0$ for $c\in[0,c_o]$.
The function $\hat{y}_2 \in C^1(c_o,\hat c)$ may be extended to $\hat{y}_2 \in C^1([0,\hat c))$ (thanks to Proposition \ref{existencey2star}) by setting $\hat{y}_2(c)=y^*_2(c)$ for $c\in[0,c_o]$. With these definitions the expression \eqref{WDayKar}, which we now recall, is valid for all $c \in [0, \hat c)$:
\begin{equation}\label{eq:exprQ}
Q(y,c)=
\left\{
\begin{array}{ll}
H(y,c), & y \in [0, \hat{y}_1(c)],\\[+6pt]
H_y(\hat{y}_2(c),c)(y - \hat{y}_2(c)) + H(\hat{y}_2(c),c), & y \in (\hat{y}_1(c), \hat{y}_2(c)),\\[+6pt]
H(y,c), & y \in [\hat{y}_2(c),\infty),
\end{array}
\right.
\end{equation}
Note that by construction and thanks to the regularity of $H$ and of the boundaries (cf.~Lemma \ref{lem:Hreg} and Proposition \ref{existencey2star}) $Q$ is well defined across $c_o$ and is continuous on $(0,\infty)\times[0,\hat{c})$. We next confirm that $Q$ is continuously differentiable.

\begin{prop}\label{prop:Q}
The function $Q$ lies in $C^1((0,\infty)\times [0, \hat c))$.
\end{prop}
\begin{proof}
Denoting $$A(y,c):=H_y(\hat{y}_2(c),c)(y-\hat{y}_2(c))+H(\hat{y}_2(c),c),$$
the surfaces $A$ and $H$ are clearly $C^1$ on $(0,\infty)\times[0,\hat{c})$ since $\hat{y}_2\in C^1([0,\hat c))$.
As a consequence $Q$ is $C^1$ away from the free boundaries $\hat{y}_1(c)$ and $\hat{y}_2(c)$ and it remains to verify whether the pasting across the boundaries is $C^1$ as well. At the two boundaries we clearly have {(cf.\ Proposition \ref{existencegeometric})}
\begin{align}\label{simple}
&H(\hat{y}_1(c),c)=A(\hat{y}_1(c),c),\quad\text{$c\in(c_o,\hat{c})$ \,\,and}\quad H(\hat{y}_2(c),c)=A(\hat{y}_2(c),c),\quad\text{$c\in[0,\hat{c})$}.
\end{align}

Recall that $\hat{y}_1\in C^1(c_o,\hat{c})$, then an application of the chain rule to the {left hand side} of \eqref{simple} gives
\begin{align}\label{simple2}
H_y(\hat{y}_1(c),c)\hat{y}'_1(c)+H_c(\hat{y}_1(c),c)=A_y(\hat{y}_1(c),c)\hat{y}'_1(c)+A_c(\hat{y}_1(c),c)
\end{align}
for $c\in(c_o,\hat c)$. Hence $H_c(\hat{y}_1(c),c)=A_c(\hat{y}_1(c),c)$ for $c\in(c_o,\hat c)$ since from the construction of $Q$ we know that $Q_y=A_y=H_y$ at the two boundaries. Similar arguments also provide $H_c(\hat{y}_2(c),c)=A_c(\hat{y}_2(c),c)$ for $c\in[0,\hat c)$.
\end{proof}

For the rest of the paper we employ exclusively analytical arguments, working in the coordinate system of the original problem \eqref{problem-intro}. Using  Proposition \ref{prop:DayKar} we therefore set 
\begin{align}\label{def:gammabeta}
\hat{\beta}(c):=
F^{-1}_\lambda(\hat{y}_2(c)), \, c \in [0, \hat c)\quad\text{and } \;
\hat{\gamma}(c):=
\left\{
\begin{array}{ll}
-\infty, & c \in [0, c_o],\\[+6pt]
F^{-1}_\lambda(\hat{y}_1(c)), & c \in  (c_o,\hat c)
\end{array}
\right.
\end{align}
and obtain the following expression for $V$:
\begin{equation}
\label{VDayKar-bis}
V(x,c)=
\left\{
\begin{array}{ll}
G(x,c), & x \in (-\infty, \hat{\gamma}(c)]\\[+6pt]
\phi_{\lambda}(x)\Big[H_y(F_\lambda(\hat{\beta}(c)),c)\Big(F_\lambda(x) - F_\lambda(\hat{\beta}(c))\Big) + H(F_\lambda(\hat{\beta}(c)),c)\Big], & x \in (\hat{\gamma}(c), \hat{\beta}(c))\\[+6pt]
G(x,c), & x \in [\hat{\beta}(c),\infty).
\end{array}
\right.
\end{equation}

\begin{remark}
\label{rem:day}
For $c\in(c_o,\hat{c})$ note that $\hat{y}_1$ and $\hat{y}_2$ solve \eqref{system-y1y2} and the second expression in \eqref{eq:exprQ} may be equivalently rewritten in terms of $\hat{y}_1$, i.e.~$Q(y,c)=H_y(\hat{y}_1(c),c)(y - \hat{y}_1(c)) + H(\hat{y}_1(c),c)$ for $y \in (\hat{y}_1(c), \hat{y}_2(c))$. Analogously \eqref{VDayKar-bis} may be equivalently rewritten in terms of $\hat{\gamma}$, that is $V(x,c)=\phi_{\lambda}(x)\Big[H_y(F_\lambda(\hat{\gamma}(c)),c)\Big(F_\lambda(x) - F_\lambda(\hat{\gamma}(c))\Big) + H(F_\lambda(\hat{\gamma}(c)),c)\Big]$ for $x \in (\hat{\gamma}(c), \hat{\beta}(c))$.
\end{remark}

\begin{coroll}
\label{cor:regBdr}
We have
\begin{itemize}
\item[ i)] The boundary $\hat{\beta}$ lies in $C^1([0,\hat{c}))$ and is strictly decreasing with $\hat{\beta}(c)\in (0,1/\sqrt{2\lambda})$ for all $c\in[0,\hat{c})$;
\item[ii)] The boundary $\hat{\gamma}$ lies in $C^1((c_o,\hat{c}])$ and is strictly increasing with $\hat{\gamma}(c)\le -1/\sqrt{2\lambda}$ for all $c\in[0,\hat{c})$.
\end{itemize}
\end{coroll}
\begin{proof}
This follows immediately from Propositions \ref{monboundaries} and \ref{existencey2star}, and \eqref{def:gammabeta}.
\end{proof}

We can now show that the parameter-dependent optimal stopping value function $V$ satisfies the following free boundary problem. This will in turn establish some properties of $W^1(x,c)=x\Phi(c) + V(x,c)$ required to verify optimality in Section \ref{verification}.

\begin{prop}
\label{prop:C1}
The value function $V$ of \eqref{def-V} belongs to $C^1(\RR\times[0,\hat{c}))$ with $V_{xx}\in L^\infty_{loc}(\RR\times(0,\hat{c}))$. Moreover $V\le G$ and satisfies
\begin{align}\label{fbp-V00}
\left\{
\begin{array}{ll}
\tfrac{1}{2}V_{xx}(x,c)-\lambda V(x,c)=0 & \text{for $\hat{\gamma}(c)<x<\hat{\beta}(c)$, $c\in[0,\hat{c})$}\\[+4pt]
\tfrac{1}{2}V_{xx}(x,c)-\lambda V(x,c)\ge0 & \text{for a.e.~$(x,c)\in\RR\times[0,\hat{c})$}\\[+4pt]
V(x,c)=G(x,c)& \text{for $x\le \hat{\gamma}(c)$, $x\ge \hat{\beta}(c)$, $c\in[0,\hat{c})$}\\[+4pt]
V_x(x,c)=G_x(x,c)& \text{for $x\le \hat{\gamma}(c)$, $x\ge \hat{\beta}(c)$, $c\in[0,\hat{c})$}\\[+4pt]
V_c(x,c)=G_c(x,c)& \text{for $x\le \hat{\gamma}(c)$, $x\ge \hat{\beta}(c)$, $c\in[0,\hat{c})$}.
\end{array}
\right.
\end{align}
\end{prop}
\begin{proof}
From Proposition \ref{prop:DayKar} we have that $Q\in C^1((0,\infty)\times [0, \hat c))$ implies $V\in C^1(\RR\times[0,\hat{c}))$. Analogously to prove that $V_{xx}$ is locally bounded it suffices to show it for $Q_{yy}$. Since $Q_{yy}=H_{yy}$ for $x\le \hat{\gamma}(c)$, $x\ge \hat{\beta}(c)$, $c\in[0,\hat{c})$ and $Q$ is linear in $y$ elsewhere the claim follows.

By construction $Q\le H$ and therefore $V\le G$. From \eqref{VDayKar-bis} we see that inside the continuation region $V$ may be rewritten as $V(x,c)=A(c)\psi_\lambda(x)+B(c)\phi_\lambda(x)$, with suitable $A(c)$ and $B(c)$, and therefore the first equation of \eqref{fbp-V00} holds. Outside the continuation region one has $V=G$ so that $\tfrac{1}{2}V_{xx}-\lambda V$ can be computed explicitly by recalling the expression for $W^o$ (see \eqref{Wo}) and it may be verified that the subsequent inequality
holds (using that $k(c)<0$ since $c<\hat{c}$ and $R(c)>0$ for $c>c_o$). The last three equalities in \eqref{fbp-V00} follow since $V\in C^1(\RR\times[0,\hat{c}))$.
\end{proof}

\begin{coroll}
\label{corW1}
$W^1\in C^1(\RR\times[0,\hat{c}))$, with $W^1_{xx}\in L^\infty_{loc}(\RR\times(0,\hat{c}))$ and in particular we have
\begin{align}
\label{eq:W1C1}W^1_c(x,c) = -x \quad \text{and}\quad W^1_x(x,c) &= \hat{c}-c+W^o_x(x,\hat{c})
\end{align}
for $x\in(-\infty,\hat{\gamma}(c)]\cup[\hat{\beta}(c),+\infty)$ and $c \in [0,\hat{c})$.
\end{coroll}
{The next two propositions follow from results collected above and their detailed proofs are given in Appendix \ref{someproofs}.}
\begin{prop}
\label{prop-Wcgeqx}
$W^1_c(x,c) \geq -x$ for all $(x,c) \in \RR\times[0,\hat{c})$.
\end{prop}
\begin{prop}
\label{prop:C1past}
Let
\begin{align}
\label{def:W}
W(x,c):=\left\{
\begin{array}{ll}
W^1(x,c), & \text{for $(x,c)\in\RR\times[0,\hat{c})$}\\[+4pt]
W^o(x,c), & \text{for $(x,c)\in\RR\times[\hat{c},1]$,}
\end{array}
\right.
\end{align}
then $W\in C^1(\RR\times[0,1])$ and $W_{xx}\in L^\infty_{loc}(\RR\times[0,1])$.
\end{prop}

In the next definition we extend the boundaries $\hat{\beta}$ and $\hat{\gamma}$ to the whole of $[0,1]$. With this extension they correspond to the boundaries introduced in the statement of Theorem \ref{thm:main-nu}.

\begin{definition}
The function $\hat{\gamma}$ can be extended to $[\hat{c},1]$ by putting $\hat{\gamma}(c)=\gamma^o$, $c\in[\hat{c},1]$ (by $(1)$\ and $(4)$\ of Proposition \ref{monboundaries} and \eqref{def:gammabeta}). This extension is $C^1$ on $(c_o,1]$ and will be assumed in the rest of the paper. We also set $\hat{\beta}(c)=+\infty$ for $c\in[\hat{c},1]$.
\end{definition}


\subsection{Step 6: verification theorem and the optimal control}
\label{verification}

In this section we establish the optimality of the candidate value function $W$ and show that the purely discontinuous control $\nu^*$ defined in \eqref{op-contr01} of Theorem \ref{thm:main-nu} is indeed optimal for problem \eqref{valuefunction}. Firstly, several results obtained above are summarised in the following proposition.

\begin{prop}
\label{prop:VI}
The function $W$ of \eqref{def:W} solves the variational problem \eqref{eq:HJB}.
Moreover, $|W(x,c)| \leq K(1 + |x|)$ or some $K>0$ and $W(x,1)=U(x,1)=0$.
\end{prop}
\begin{proof}
The functions $W^o$ and $W^1$ solve the variational problem on $\RR\times[\hat{c},1)$ and $\RR\times(0,\hat{c})$ respectively (see Proposition \ref{prop:solWo} for the claim regarding $W^o$ and Propositions \ref{prop:C1}, \ref{prop-Wcgeqx} and Corollary \ref{corW1} for the claim regarding $W^1$). Then Proposition \ref{prop:C1past} guarantees that $W$ solves the variational problem on $\RR\times(0,1)$ as required.
From the definitions of $W$, $W^o$ and $W^1$ (see \eqref{def:W}, \eqref{Wo} and \eqref{Wstar}) one also obtains the sublinear growth property and $W(x,1)=0$.
\end{proof}

\begin{theorem}
\label{thm-opt-c}
The admissible control $\nu^*$ of \eqref{op-contr01} is optimal for problem \eqref{valuefunction} and $W\equiv U$.
\end{theorem}
\begin{proof}
The proof is based on a verification argument and, as usual, it splits into two parts.\vspace{+8pt}

(i) 
Fix $(x,c)\in\RR\times[0,1]$ and take $R>0$. Set $\tau_{R}:=\inf\big\{t\ge 0\,:\,X^x_t\notin(-R,R)\big\}$, take an admissible control $\nu$, and recall the regularity results for $W$ in Proposition \ref{prop:C1past}. Then we can use It\^o's formula in the weak version of \cite{FlemingSoner}, Chapter 8, Section VIII.4, Theorem 4.1, up to the stopping time $\tau_R \wedge T$, for some $T>0$, to obtain
\begin{align*}
W(x,c) =& \EE\left[e^{-\lambda(\tau_{R}\wedge T)}W(X_{\tau_{R}\wedge T}^x, {C}_{\tau_{R}\wedge T}^{c,\nu})\right]-\EE\bigg[\int_0^{\tau_{R}\wedge T}e^{-\lambda s}(\tfrac{1}{2}W_{xx}-\lambda W)(X^x_s,{C}^{c,\nu}_s)ds\bigg] \\
&  - \EE\bigg[\int_0^{\tau_{R}\wedge T}e^{-\lambda s}W_c(X^x_s,{C}^{c,\nu}_s)d\nu_s\bigg]
\\ &  - \EE\Big[\sum_{0\leq s < \tau_{R}\wedge T}e^{-\lambda s}
\left(W(X^x_s,{C}^{c,\nu}_{s+})-W(X^x_s,{C}^{c,\nu}_s)-W_c(X^x_s,{C}^{c,\nu}_s)\Delta \nu_s\right)\Big]
\end{align*}
where $\Delta \nu_s := \nu_{s+}-\nu_s$ and the expectation of the stochastic integral vanishes since $W_x$ is bounded on $(x,c)\in[-R,R]\times[0,1]$.

Now, recalling that any $\nu \in \mathcal{A}_c$ can be decomposed into the sum of its continuous part and its pure jump part, i.e.\ $d\nu=d\nu^{cont}+ \Delta \nu$, one has (see \cite{FlemingSoner}, Chapter 8, Section VIII.4, Theorem 4.1 at pp.\ 301-302)
\begin{align*}
W(x,c) = & \EE\left[e^{-\lambda(\tau_{R}\wedge T)}W(X_{\tau_{R}\wedge T}^x, C_{\tau_{R}\wedge T}^{c,\nu})\right]-\EE\bigg[\int_0^{\tau_{R}\wedge T}e^{-\lambda s}(\tfrac{1}{2}W_{xx}-\lambda W)(X^x_s,C^{c,\nu}_s)ds\bigg] \\
&  - \EE\bigg[\int_0^{\tau_{R}\wedge T}e^{-\lambda s}W_c(X^x_s,C^{c,\nu}_s)d\nu_s^{cont} + \sum_{0\le s < \tau_{R}\wedge T}e^{-\lambda s}
\left(W(X^x_s,C^{c,\nu}_{s+})-W(X^x_s,C^{c,\nu}_s)\right)\bigg].
\end{align*}
Since $W$ satisfies the Hamilton-Jacobi-Bellman equation \eqref{eq:HJB} (cf.\ Proposition \ref{prop:VI}) and by {writing}
\begin{equation}
\label{jump}
W(X^x_s,C^{c,\nu}_{s+})-W(X^x_s,C^{c,\nu}_{s})=
\int_0^{\Delta \nu_s} W_c(X^x_s,C^{c,\nu}_{s}+u)du,
\end{equation}
we obtain
\begin{align}
\label{verif04}
W(x,c) \leq &\EE\left[e^{-\lambda(\tau_{R}\wedge T)}W(X_{\tau_{R}\wedge T}^x, C_{\tau_{R}\wedge T}^{c,\nu})\right]+\EE\bigg[\int_0^{\tau_{R}\wedge T}e^{-\lambda s}\lambda X^x_s \Phi(C^{c,\nu}_s)ds\bigg]\nonumber \\
& + \EE\bigg[\int_0^{\tau_{R}\wedge T}e^{-\lambda s} X^x_s d\nu_s^{cont}\bigg]
 + \EE\Big[\sum_{0\leq s < \tau_{R}\wedge T}e^{-\lambda s}
X^x_s \Delta \nu_s \Big] \\
 = &\EE\bigg[e^{-\lambda(\tau_{R}\wedge T)}W(X_{\tau_{R}\wedge T}^x, C_{\tau_{R}\wedge T}^{c,\nu})+\int_0^{\tau_{R}\wedge T}e^{-\lambda s}
\lambda X^x_s \Phi(C^{c,\nu}_s)ds + \int_0^{\tau_{R}\wedge T}e^{-\lambda s}X^x_s d\nu_s\bigg].\nonumber
\end{align}

When taking limits as $R\to\infty$ we have $\tau_{R}\wedge T \rightarrow T$, $\PP$-a.s. By standard properties of Brownian motion it is easy to prove that the integral terms in the last expression on the right hand side of \eqref{verif04} are uniformly bounded in $L^2(\Omega,\PP)$, hence uniformly integrable. Moreover, $W$ has sub-linear growth by Proposition \ref{prop:VI}. Then we also take limits as $T\uparrow \infty$ and it follows that
\begin{align}
W(x,c) \leq \EE\bigg[\int_0^{\infty}e^{-\lambda s}\lambda X^x_s \Phi(C^{c,\nu}_s)ds + \int_0^{\infty}e^{-\lambda s} X^x_s d\nu_s \bigg],
\end{align}
due to the fact that $\lim_{T \rightarrow \infty}\EE[e^{-\lambda T}W(X^x_T,C^{c,\nu}_T)]=0$.
Since the latter holds for all admissible $\nu$ we have $W(x,c) \leq U(x;c)$.
\vspace{+8pt}

(ii)
If $c=1$ then $W(x,1)=0=U(x,1)$. Then take $c\in[0,1)$ and define $C^*_t:=C^{c,\nu^*}_t=c+\nu^*_t$, with $\nu^*$ as in \eqref{op-contr01}. Applying It\^o's formula again (possibly using localisation arguments as above) up to time $\tau_{\hat{\gamma}}$ (cf.\ \eqref{taubetataugamma}) we find
\begin{align}
\label{ito}
W(x,c)=&\EE\bigg[e^{-\lambda\tau_{\hat{\gamma}}}W(X^x_{\tau_{\hat{\gamma}}},C^*_{\tau_{\hat{\gamma}}}) - \int^{\tau_{\hat{\gamma}}}_0 e^{-\lambda s} (\tfrac{1}{2}W_{xx}-\lambda W)(X^x_s,C^*_{s}) ds\bigg]\nonumber\\
&-\EE\bigg[\sum_{0\le s<\tau_{\hat{\gamma}}}e^{-\lambda s}\big(W(X^x_s,C^*_{s+})-W(X^x_s,C^*_{s})\big)\bigg],
\end{align}
where we have used that $\nu^*$ does not have a continuous part. We also recall, as already observed, that $\tau_{\hat{\gamma}}<+\infty$, $\PP$-a.s.~under the control policy of $\nu^*$.

From \eqref{tausigmastar} one has $\tau^*\le\tau_{\hat{\gamma}}$, $\PP$-a.s.~and therefore we can always write
\begin{align}\label{eq:int00}
\int^{\tau_{\hat{\gamma}}}_0 &e^{-\lambda s} (\tfrac{1}{2}W_{xx}-\lambda W)(X^x_s,C^*_{s}) ds\nonumber\\
=&  \int^{\tau^*}_0 e^{-\lambda s} (\tfrac{1}{2}W_{xx}-\lambda W)(X^x_s,C^*_{s}) ds
+ \int^{\tau_{\hat{\gamma}}}_{\tau^*} e^{-\lambda s} (\tfrac{1}{2}W_{xx}-\lambda W)(X^x_s,C^*_{s}) ds\nonumber\\
=& - \int^{\tau^*}_0 e^{-\lambda s} \lambda X^x_s\Phi(C^*_{s}) ds + \int^{\tau_{\hat{\gamma}}}_{\tau^*} e^{-\lambda s} (\tfrac{1}{2}W_{xx}-\lambda W)(X^x_s,C^*_{s}) ds
\end{align}
where the last equality follows by recalling that $(\tfrac{1}{2}W_{xx}-\lambda W)(x,c)=-\lambda x\Phi(c)$ for $\hat{\gamma}(c)<x<\hat{\beta}(c)$ and hence it holds in the first integral for $s\le\tau^*$. To evaluate the last term of \eqref{eq:int00} we study separately the events $\{\tau^* = \tau_{\hat{\beta}}\}$ and $\{\tau^* = \tau_{\hat{\gamma}}\}$. We start by observing that under the control strategy $\nu^*$ one has $\{\tau^* = \tau_{\hat{\beta}}\}= \{\tau_{\hat{\gamma}}=\sigma^*\}$ and we get
\begin{align}\label{eq:int01}
\mathds{1}_{\{\tau^* = \tau_{\hat{\beta}}\}}\int^{\tau_{\hat{\gamma}}}_{\tau^*} e^{-\lambda s} (\tfrac{1}{2}W_{xx}-\lambda W)(X^x_s,C^*_{s}) ds = - \mathds{1}_{\{\tau^* = \tau_{\hat{\beta}}\}} \int^{\tau_{\hat{\gamma}}}_{\tau^*} e^{-\lambda s} \lambda X^x_s\Phi(C^*_{s}) ds
\end{align}
by Proposition \ref{prop:solWo} since $(X^x_s, C^*_s) = (X^x_s,\hat{c})$ for any $\tau^* < s \leq \tau_{\hat{\gamma}}=\sigma^*$ on $\{\tau^* = \tau_{\hat{\beta}}\}$. On the other hand
\begin{align}\label{eq:int02}
\mathds{1}_{\{\tau^* = \tau_{\hat{\gamma}}\}}\int^{\tau_{\hat{\gamma}}}_{\tau^*} e^{-\lambda s} (\tfrac{1}{2}W_{xx}-\lambda W)(X^x_s,C^*_{s}) ds = 0 = \mathds{1}_{\{\tau^* = \tau_{\hat{\gamma}}\}}\int^{\tau_{\hat{\gamma}}}_{\tau^*} e^{-\lambda s} \lambda X^x_s\Phi(C^*_{s}) ds.
\end{align}
Then it follows from \eqref{eq:int00}, \eqref{eq:int01} and \eqref{eq:int02} that
\begin{align}
\label{observation1}
 \int^{\tau_{\hat{\gamma}}}_0 e^{-\lambda s} (\tfrac{1}{2}W_{xx}-\lambda W)(X^x_s,C^*_{s}) ds = - \int_{0}^{\tau_{\hat{\gamma}}}\lambda X^x_s \Phi(C^*_s) ds.
\end{align}
Moreover $\Phi(C^*_s)=0$ for any $s > \tau_{\hat{\gamma}}$ because $C^*_s = 1$ for any such $s$ and thus we finally get from \eqref{observation1}
\begin{equation}
\label{observation2}
\int^{\tau_{\hat{\gamma}}}_0 e^{-\lambda s} (\tfrac{1}{2}W_{xx}-\lambda W)(X^x_s,C^*_{s}) ds =  -\int_{0}^{\infty}\lambda X^x_s \Phi(C^*_s) ds.
\end{equation}

Note that under the control strategy $\nu^*$ {we also have} $\{\tau_{\hat{\gamma}}<\tau_{\hat{\beta}}\}=\{\tau_{\hat{\gamma}}<\sigma^*\}$ and $\{\tau_{\hat{\gamma}}>\tau_{\hat{\beta}}\}=\{\tau_{\hat{\gamma}}=\sigma^*\}$, then from \eqref{op-contr01} we have
\begin{align}
\label{observation3}
\EE&\left[e^{-\lambda\tau_{\hat{\gamma}}}W(X^x_{\tau_{\hat{\gamma}}},C^*_{\tau_{\hat{\gamma}}})\right] \nonumber \\ =&\EE\left[\mathds{1}_{\{\tau_{\hat{\gamma}}>\tau_{\hat{\beta}}\}}e^{-\lambda\tau_{\hat{\gamma}}}W(\hat{\gamma}(\hat{c}),\hat{c})\right]
+\EE\left[\mathds{1}_{\{\tau_{\hat{\gamma}}<\tau_{\hat{\beta}}\}}e^{-\lambda\tau_{\hat{\gamma}}}W(\hat{\gamma}(c),c)\right] \nonumber \\ =&\EE\left[\mathds{1}_{\{\tau_{\hat{\gamma}}>\tau_{\hat{\beta}}\}}e^{-\lambda\tau_{\hat{\gamma}}}\hat{\gamma}(\hat{c})(1-\hat{c})\right]+ \EE\left[\mathds{1}_{\{\tau_{\hat{\gamma}}<\tau_{\hat{\beta}}\}}e^{-\lambda\tau_{\hat{\gamma}}}\hat{\gamma}(c)(1-c)\right]\nonumber \\
=&\EE\bigg[ \int_{\tau_{\hat{\gamma}}}^{\infty} e^{-\lambda s} X^x_s d\nu^*_s \bigg]
\end{align}
by using that $W(\hat{\gamma}(c), c) =\hat{\gamma}(c)(1-c)$ for all $c\in[0,1)$ as proved in Section \ref{construction} (see also Figure \ref{fig:1}).

For the jump part of the control, i.e.~for the last term in \eqref{ito}, again we argue in a similar way as above and use that on the event $\{\tau^*=\tau_{\hat{\gamma}}\}$ there is no jump strictly prior to $\tau_{\hat{\gamma}}$ and the sum in \eqref{ito} is zero, whereas on the event $\{\tau^*=\tau_{\hat{\beta}}\}$ a single jump occurs prior to $\tau_{\hat{\gamma}}$, precisely at $\tau_{\hat{\beta}}$. This gives
\begin{align}\label{jump1}
\EE&\Big[\sum_{0 \leq s < \tau_{\hat{\gamma}}}e^{-\lambda s}\big(W(X^x_s,C^*_{s+})-W(X^x_s,C^*_{s})\big)\Big]\nonumber\\
=&\EE\Big[\mathds{1}_{\{\tau^*=\tau_{\hat{\gamma}}\}}\cdot 0+\mathds{1}_{\{\tau^*=\tau_{\hat{\beta}}\}}e^{-\lambda\tau_{\hat{\beta}}}X^x_{\tau_{\hat{\beta}}}(\hat{c}-c)\Big]=\EE\bigg[\int^{\tau_{\hat{\gamma}}}_0 e^{-\lambda s} X^x_s d\,\nu^{*}_s\bigg].
\end{align}
Combining \eqref{observation2}, \eqref{observation3} and \eqref{jump1} it follows from \eqref{ito} that
\begin{equation}
\label{opt-C02}
W(x,c)=\EE\bigg[\int^{\infty}_0 e^{-\lambda s}\lambda\,X^x_s\,\Phi(C^*_s) ds + \int_0^{\infty}e^{-\lambda s}X^x_s d\nu^*_s\bigg] \geq U(x,c),
\end{equation}
which together with (i) above implies $W(x,c)=U(x,c)$, $(x,c)\in\RR\times[0,1]$ and $\nu^*$ of \eqref{op-contr01} is optimal.
\end{proof}

\begin{remark}
\label{rem:opt}
Note that, unusually, from \eqref{Wo} we see $W^o_c(x,\hat{c})=-x$ for all $x\in\RR$ (see Figure \ref{fig:1}) whereas we would expect $W^o_c(x,\hat{c})>-x$ for $x>\gamma^o$. This is possible because $(\tfrac{1}{2}W_{xx}-\lambda W)(x,\hat{c})=-\lambda\Phi(\hat{c})x$ for $x>\gamma^o$ and therefore, as long as $X$ stays above $\gamma^o$, an inaction strategy does not increase the overall costs. \end{remark}


\appendix
\section{Some proofs needed in Section \ref{construction}}
\label{someproofs}
\renewcommand{\theequation}{A-\arabic{equation}}

\begin{proof}{\textbf{[Proposition \ref{monboundaries}]}}\vspace{+5pt}

Rewrite \eqref{system-y1y2} as
\begin{equation}
\label{system-y1y2-tris}
F_1(\hat{y}_1(c),\hat{y}_2(c);c) =0\,\quad\text{and}\quad F_2(\hat{y}_1(c),\hat{y}_2(c);c) =0,
\end{equation}
with the two functions $F_i: (0,\infty) \times (0,\infty) \times [0,1]$, $i=1,2$, defined by \eqref{F1} and \eqref{F2} of Proposition \ref{thm:main-bd}.
The Jacobian matrix
\begin{align*}
J(x,y,c)=\left[ \begin{array}{ll}
\pd{F_1}{x}& \pd{F_1}{y}\\[+3pt]
\pd{F_2}{x} & \pd{F_2}{y}\end{array}\right](x,y,c)
=\frac 1 4 \left[ \begin{array}{ll}
-R(c) x^{-\frac 3 2 }\ln x &  (R(c)-R(\hat c))y^{-\frac 3 2 }\ln y\\[+3pt]
-R(c)x^{-\frac 1 2 }\ln x &  (R(c)-R(\hat c)) y^{-\frac 1 2 }\ln y
\end{array}\right]
\end{align*}
 has determinant
\[
D(x,y,c):= \frac{1}{16}\big[R(c) - R(\hat{c})\big]R(c)\frac{1}{\sqrt{xy}}{\Big(\frac{1}{y} - \frac{1}{x}\Big)}\ln x\ln y
\]
which is strictly negative when $x \le e^{-2}$, $y > 1$ and $c \in(c_0, \hat c)$. To simplify notation we suppress the dependency of $\hat{y}_1$ and $\hat{y}_2$ on $c$. Total differentiation of \eqref{system-y1y2-tris}
with respect to $c$ and the Implicit Function Theorem imply that $\hat{y}_1$ and $\hat{y}_2$ lie in $C^1(c_o,\hat c)$, with
\begin{align*}
\left[ \begin{array}{l} \hat{y}_1'(c) \\ \hat{y}_2'(c) \end{array}\right] = J^{-1}
\left[ \begin{array}{l} -\pd{F_1}{c} \\[+3pt] -\pd{F_1}{c} \end{array}\right](\hat{y}_1,\hat{y}_2,c)
= R'(c) J^{-1}(\hat{y}_1,\hat{y}_2,c) \left[ \begin{array}{l}
\hat{y}_1^{-\frac{1}{2}}(1 + \tfrac{1}{2}\ln \hat{y}_1) - \hat{y}_2^{-\frac{1}{2}}(1 + \tfrac{1}{2}\ln \hat{y}_2) \\
\hat{y}_1^{\frac{1}{2}}(1 - \tfrac{1}{2}\ln \hat{y}_1) - \hat{y}_2^{\frac{1}{2}}(1 - \tfrac{1}{2}\ln \hat{y}_2)
 \end{array}\right]
\end{align*}
so that
\begin{align}
\label{derivativesy1}
\hat{y}_1'(c)=&\frac{1}{4}\frac{R'(c)\big[R(c) - R(\hat{c})\big]}{D(\hat{y}_1,\hat{y}_2,c)}\Big[\sqrt{\frac{\hat{y}_1}{\hat{y}_2}}(1 - \frac{1}{2}\ln \hat{y}_1) + \ln \hat{y}_2 - \sqrt{\frac{\hat{y}_2}{\hat{y}_1}}(1 + \frac{1}{2}\ln \hat{y}_1 )\Big]\hat{y}_2^{-1}\ln \hat{y}_2 \nonumber\\[+4pt]
=:&D_1(\hat{y}_1,\hat{y}_2,c)/D(\hat{y}_1,\hat{y}_2,c)
\end{align}
and
\begin{align}
\label{derivativesy2}
\hat{y}_2'(c)=&-\frac{1}{4}\frac{R(c)R'(c)}{D(\hat{y}_1,\hat{y}_2,c)} \Big[\hat{y}_1^{\frac{1}{2}}\ln \hat{y}_1 - \hat{y}_1\hat{y}_2^{-\frac{1}{2}}(1 + \frac{1}{2}\ln \hat{y}_2) + \hat{y}_2^{\frac{1}{2}}(1 - \frac{1}{2}\ln \hat{y}_2)\Big]
\hat{y}_1^{-\frac{3}{2}}\ln \hat{y}_1\nonumber\\[+4pt]
=:&D_2(\hat{y}_1,\hat{y}_2,c)/D(\hat{y}_1,\hat{y}_2,c)
\end{align}
It is not hard to verify that $\hat{y}'_1(c)>0$ by using $0<\hat{y}_1<e^{-2}$, $\hat{y}_2>1$ and $R(c)<R(\hat{c})$.
The sign of the right hand side of \eqref{derivativesy2} is opposite to the sign of
\begin{equation}
\label{def-hatD2}
\hat{D}:=\hat{y}_1^{\frac{1}{2}}\ln \hat{y}_1 - \hat{y}_1\hat{y}_2^{-\frac{1}{2}}(1 + \frac{1}{2}\ln \hat{y}_2) +
\hat{y}_2^{\frac{1}{2}}(1 - \frac{1}{2}\ln \hat{y}_2)
\end{equation}
since $\ln \hat{y}_1<0$, $R'(c) >0$ and $D(\hat{y}_1,\hat{y}_2,c)<0$, $c\in(c_0,\hat{c})$.
Recalling now \eqref{system-y1y2-tris}, \eqref{F1} and \eqref{F2}, we obtain
\begin{align*}
\hat{y}_2^{\frac{1}{2}}\big(1 - \frac{1}{2}\ln\hat{y}_2\big) = \frac{R(c)}{R(c)-R(\hat{c})}\hat{y}_1^{\frac{1}{2}}\big(1 - \frac{1}{2}\ln\hat{y}_1\big) - \frac{2e^{-1}}{R(c)-R(\hat{c})}R(\hat{c}),
\end{align*}
and
\begin{align*}
\hat{y}_2^{-\frac{1}{2}}\big(1 + \frac{1}{2}\ln \hat{y}_2\big) = \frac{R(c)}{R(c)-R(\hat{c})}\hat{y}_1^{-\frac{1}{2}}\big(1 + \frac{1}{2}\ln\hat{y}_1\big),
\end{align*}
which plugged into \eqref{def-hatD2} give
\begin{align*}
\hat{D}= - \frac{R(\hat{c})}{(R(c)-R(\hat{c})}(\sqrt{\hat{y}_1}\ln\hat{y}_1 + 2e^{-1})=:- \frac{R(\hat{c})}{(R(c)-R(\hat{c})}q(\hat{y}_1).
\end{align*}
It is now easy to see that $x \mapsto q(x)$ is strictly decreasing on $(0,e^{-2})$ and such that $q(e^{-2})=0$ and $\lim_{x \downarrow 0}q(x) = 2e^{-1} > 0$. Hence $q(\hat{y}_1) > 0$ implies that $\hat{D}>0$ and $\hat{y}_2^{'}(c) <0$.

To complete the proof we need to show properties $(1)$-$(4)$. We observe that due to the monotonicity of $\hat{y}_i(\cdot)$, $i=1,2$, on $(c_o, \hat{c})$ their limits exist at all points of this interval. 
\begin{itemize}
\item [$(1)$]
Taking limits as $c \uparrow \hat{c}$ in the second equation of \eqref{system-y1y2-tris}, using \eqref{F2} and defining $\hat{y}_1(\hat{c}-):=\lim_{c \uparrow \hat{c}}\hat{y}_1(c)$ we get
\begin{align*}
\hat{y}_1^{\frac{1}{2}}(\hat{c}-)\big(1 - \frac{1}{2}\ln\hat{y}_1(\hat{c}-)\big) = 2e^{-1},
\end{align*}
which is uniquely solved by $\hat{y}_1(\hat{c}-)=e^{-2}$.
\item[$(2)$] We argue by contradiction and assume that $\lim_{c\downarrow c_o}\hat{y}_1(c) =\overline{y}_1 > 0$. Then taking limits as $c\downarrow c_o$ in the first equation of \eqref{system-y1y2-tris} and recalling that $R(c_o)=0$ we find
\begin{align*}
R(\hat{c})\sqrt{\hat{y}_2(c_o+)}\big[1 + \frac{1}{2}\ln\hat{y}_2(c_o+)\big]=0,
\end{align*}
which is clearly impossible since $\hat{y}_2(c_o+)\geq 1$ due to the fact that $\hat{y}_2(c)>1$ for any $c\in (c_o,\hat{c})$.
\item [$(3)$] From the second equality in \eqref{system-y1y2-tris} and by \eqref{F2} one finds
\begin{equation}
\label{stima4}
\sqrt{\hat{y}_2(c)}\big[1 - \frac{1}{2}\ln\hat{y}_2(c)\big] = \frac{2e^{-1}R(\hat{c}) - R(c)\sqrt{\hat{y}_1(c)}\big[1 - \frac{1}{2}\ln\hat{y}_1(c)\big]}{R(\hat{c}) -R(c)}\ge 2 e ^{-1} > 0
\end{equation}
where the first lower bound follows by the fact that $x \mapsto \sqrt{x}[1 - \frac{1}{2}\ln(x)]$ is strictly increasing and positive on $[0,e^{-2}]$, with maximum value $2e^{-1}$. Since also $\hat{y}_2(c)>0$, from \eqref{stima4} we conclude that $\big[1 - \frac{1}{2}\ln\hat{y}_2(c)\big] > 0$, thus implying $\hat{y}_2(c)<e^{2}$.
\item[$(4)$]
We take limits as $c \uparrow \hat{c}$ in \eqref{derivativesy1} and notice that $\hat{y}'_1(c)\to 0$ since $R'(c)\to 0$ (notice also that both functions $D_1$ and $D$ are proportional to $\ln \hat{y}_2$ so that their quotient remains finite).
\end{itemize}
\end{proof}

\begin{proof}{\textbf{[Proposition \ref{existencey2star}]}}\vspace{+5pt}

Note first that by \eqref{def-H2} and \eqref{def-Hy}, equation \eqref{defy2star} may be rewritten in the equivalent form
\begin{equation}
\label{defy2star-bis}
F_3(y;c)=0,
\end{equation}
where the jointly continuous function $F_3:(0,\infty) \times [0,1] \mapsto \mathbb{R}$ is defined in \eqref{F3} of Proposition \ref{thm:main-bd}.
The proof is carried out in three parts and for simplicity we omit the dependency of $y^*_2$ on $c$.\vspace{0.1cm}

(i)
 {The function $f(y):=\sqrt{y}(1-\frac{1}{2}\ln(y))$
is strictly decreasing on $(1,e^{2})$ with $f(1)=1$ and $f(e^2)=0$ so, since the absolute value of the second term of \eqref{F3} is smaller than one,} there exists a unique $y_2^*(c) \in (1,e^2)$ solving \eqref{defy2star-bis}. Moreover since
\begin{align}
\frac{\partial F_3}{\partial y}(y,c)=-\tfrac{1}{4} y^{-\frac{1}{2}}\ln y<0\quad\text{for $(y,c)\in (1,e^2)\times [0,c_o)$}
\end{align}
we can use the implicit function theorem to conclude {that $y^*_2\in C^1([0,c_o))$ and}

\begin{align}\label{eq:yprime}
(y^*_2)'(c)= -\Big(\displaystyle\frac{\partial F_3}{\partial y}\Big /\displaystyle \frac{\partial F_3}{\partial c}\Big)(y^*_2,c)=-\frac{8e^{-1}R(\hat{c})R'(c)}{\big(R(\hat{c})-R(c)\big)^2}\frac{\sqrt{y^*_2}}{\ln y^*_2}<0\quad\text{for $c\in [0,c_o)$.}
\end{align}
\vspace{0.1cm}

(ii)
The limit $y^*_2(c_o-):=\lim_{c\uparrow c_o}y^*_2(c)$ exists by monotonicity and so by continuity we have $F_3(y^*_2(c_o-);c_o)=0$, i.e.
\begin{align}\label{eq:limy*2}
\sqrt{y^*_2(c_o-)}\big(1-\tfrac{1}{2}\ln y^*_2(c_o-)\big)=2e^{-1}.
\end{align}

We now take limits as $c\downarrow c_o$ in \eqref{F2} and use part \ref{y1limco} of Proposition \ref{monboundaries} to conclude that
\begin{align}\label{eq:limy*3}
\sqrt{\hat{y}_2(c_o+)}\big(1-\tfrac{1}{2}\ln \hat{y}_2(c_o+)\big)=2e^{-1},
\end{align}
where $\hat{y}_2(c_o+):=\lim_{c\downarrow c_o}\hat{y}_2(c)$ exists by monotonicity of $\hat{y}_2$ (cf.~Proposition \ref{monboundaries}). Hence from \eqref{eq:limy*2}, \eqref{eq:limy*3} and uniqueness of the solution to $F_3(y;c_o)=0$ we obtain \eqref{limdxsx1}.
\vspace{0.1cm}

(iii)
 Setting $\hat{y}_2(c_o):=y^*_2(c_o-)=\hat{y}_2(c_o+)$ and taking limits as $c\uparrow c_o$ in \eqref{eq:yprime}
we obtain
\begin{equation}
\label{limy2prime}
(y_2^*)'(c_o-) = -\frac{8e^{-1}R'(c_o)\sqrt{\hat{y}_2(c_o)}}{R(\hat{c}) \ln \hat{y}_2(c_o)}.
\end{equation}

We now turn to study the limit of $\hat{y}^{'}_2(c)$ when $c \downarrow c_o$. 
We have $\hat{y}_1(c) \downarrow 0$ and $R(c) \downarrow 0$, however 
by taking limits in the first equation of \eqref{system-y1y2-tris} it turns out that 
\begin{align}\label{limrate}
\lim_{c\downarrow c_o}R(c) \hat{y}_1^{-\frac{1}{2}}(c) \ln \hat{y}_1(c)= -\ell\quad\text{for some $\ell>0$},
\end{align}
since $\hat{y}_2(c_o+)$ exists in $[1,e^2]$ (see Proposition \ref{monboundaries}).
Therefore as $c$ approaches $c_o$ from above we have the following asymptotic behaviours in \eqref{derivativesy2}
$$D(\hat{y}_1(c), \hat{y}_2(c),c) \approx \frac{1}{16}R(\hat{c})R(c)\hat{y}_2^{-\frac{1}{2}}(c) \ln \hat{y}_2(c) \hat{y}_1^{-\frac{3}{2}}(c) \ln \hat{y}_1(c),$$
and
$$D_2(\hat{y}_1(c), \hat{y}_2(c),c) \approx -\tfrac{1}{4} R'(c)R(c)\hat{y}_2^{\frac{1}{2}}(c)\big(1 - \tfrac{1}{2} \ln \hat{y}_2(c)\big) \hat{y}_1^{-\frac{3}{2}}(c) \ln \hat{y}_1(c).$$
Hence
\begin{equation}
\label{DooverD2}
\hat{y}'_2(c)\approx -4 \frac{R'(c)}{R(\hat{c})}\left[\frac{\hat{y}_2^{\frac{1}{2}}(c)\Big(1 - \frac{1}{2} \ln \hat{y}_2(c)\Big)}{\hat{y}_2^{-\frac{1}{2}}(c) \ln \hat{y}_2(c)}\right].
\end{equation}
and \eqref{limdxsx2} now follows from \eqref{eq:limy*3} in the limit as $c \downarrow c_o$.
\end{proof}

\begin{proof}{\textbf{[Proposition \ref{prop-Wcgeqx}]}}\vspace{+5pt}

Recalling \eqref{Wstar2}, \eqref{def-V}, \eqref{def-G} and Proposition \ref{prop:C1} we see that it suffices to show that $V_c(x,c) \geq G_c(x,c)$ for any $x \in (\hat{\gamma}(c), \hat{\beta}(c))$ and $c\in[0,\hat{c})$. The proof is performed in two parts.
\vspace{+6pt}

(i)
 Fix $c\in [0,c_o]$ and recall that (cf.\ Section \ref{clessco} and \eqref{regions}) for any such $c$ the continuation set is of the form $(-\infty,\hat{\beta}(c))$. Define $u:=V_c - G_c$, then it is not hard to see by \eqref{def-G}, Proposition \ref{prop:C1} and \eqref{fbp-V00} that $u\in C(\RR\times[0,c_o])$ and it is the unique classical solution of
\begin{equation}
\label{eq-u}
(\tfrac{1}{2}\tfrac{d^2}{dx^2} - \lambda)u(x,c) = - \lambda x (1 + \Phi'(c)),\quad\text{for $x<\hat{\beta}(c)$ with $u(\hat{\beta}(c),c)=0$.}
\end{equation}
Therefore, setting $\tau_\beta:=\inf\{t\geq 0: X^x_t \geq \hat{\beta}(c)\}$
and using the Feynmann-Kac representation formula (possibly up to a standard localisation argument), we get
\begin{align}
\label{FeyKac-u-clessc0}\hspace{-5pt}
u(x,c)  \hspace{-2pt}= &\EE\bigg[e^{-\lambda \tau_\beta}u(X^x_{\tau_\beta},c) + \lambda(1 + \Phi'(c))\hspace{-3pt}\int_0^{\tau_\beta}\hspace{-4pt}e^{-\lambda t} X^x_t dt\bigg]\hspace{-2pt}  =\hspace{-2pt}  (1 + \Phi'(c))\EE\bigg[\hspace{-3pt}\int_0^{\tau_\beta}\hspace{-4pt}\lambda e^{-\lambda t} X^x_t dt\bigg],
\end{align}
where we have used that $u(X^x_{\tau_\beta},c) = 0$ $\PP$-a.s.~since $\tau_\beta < \infty$ $\PP$-a.s.\ by the recurrence property of Brownian motion. Recalling that $X^x_t = x + B_t$ (cf.\ \eqref{statevariable}) 
Dynkin's formula and standard formulae for the Laplace transform of $\tau_\beta$ lead from \eqref{FeyKac-u-clessc0} to
\begin{equation}
\label{uclessco}
u(x,c) = (1 + \Phi'(c))\Big( x - \EE\big[e^{-\lambda \tau_\beta} X^x_{\tau_\beta}\big]\Big) = (1 + \Phi'(c))\left[ x - \hat{\beta}(c) \frac{\psi_{\lambda}(x)}{\psi_{\lambda}(\hat{\beta}(c))}\right].
\end{equation}
Since $(1+\Phi'(c)) < 0$ for $c\in [0,c_o]$, we have $u(x,c)=V_c(x,c) - G_c(x,c) \geq 0$ if and only if
\begin{equation}
\label{thetaclessco}
\theta(x,c):=x - \hat{\beta}(c) \frac{\psi_{\lambda}(x)}{\psi_{\lambda}(\hat{\beta}(c))}\leq 0\quad\text{for $x<\hat{\beta}(c)$}.
\end{equation}
From Proposition \ref{existencey2star} we obtain $1<\hat{y}_2(c) < e^2$ and hence $0<\hat{\beta}(c) < 1/\sqrt{2\lambda}$. Therefore, also recalling that $\psi_{\lambda}(x) = e^{\sqrt{2\lambda} x}$, one has for any $x < \hat{\beta}(c)$
$$\theta_x(x,c)=1 - \hat{\beta}(c)\sqrt{2\lambda}e^{\sqrt{2\lambda}(x - \hat{\beta}(c))} \geq 1 - \hat{\beta}(c)\sqrt{2\lambda} \geq 0.$$
We can now conclude that \eqref{thetaclessco} is fulfilled since $\theta(\,\cdot\,,c)$ is increasing for $x<\hat{\beta}(c)$ and $\theta(\hat{\beta}(c),c)=0$. Hence $u \geq 0$ in $(-\infty, \hat{\beta}(c)) \times [0,c_o]$.
\vspace{+6pt}

(ii)
 Fix now $c\in (c_o,\hat{c})$, take $x \in (\hat{\gamma}(c), \hat{\beta}(c))$ and denote again $u:=V_c - G_c$. As in part (i) above it is not hard to see that
\begin{align}\label{eq:u01}
(\tfrac{1}{2}\tfrac{d^2}{dx^2}-\lambda) u(x,c)=-\lambda x(1+\Phi'(c))\quad\text{for $x\in(\hat{\gamma}(c), \hat{\beta}(c))$ and $u(\hat{\gamma}(c),c)=u(\hat{\beta}(c),c)=0$.}
\end{align}
Set $\tau_{\gamma,\beta}:=\tau_\gamma\wedge\tau_\beta$ with $\tau_\beta$ as in part (i) above and $\tau_\gamma:=\inf\{t\ge0:X^x_t\le\hat{\gamma}(c)\}$.
Then $u$ is continuous and it admits the Feynmann-Kac representation
\begin{align}
\label{FeyKac-u}
u(x,c) \hspace{-2pt} = \hspace{-2pt} (1 + \Phi'(c))\EE\bigg[\int_0^{\tau_{\gamma,\beta}}\lambda e^{-\lambda t} X^x_t dt\bigg]
\end{align}
where we have used that $u(X^x_{\tau_{\gamma,\beta}},c) = 0$ $\PP$-a.s.\ due to \eqref{eq:u01} and to the fact that $\tau_{\gamma,\beta} < \infty$ $\PP$-a.s.\ by the recurrence property of Brownian motion. Since $(1 + \Phi'(c)) <0 $ on $[c_o,\hat{c})$ then $u(x,c) \geq 0$ on $(\hat{\gamma}(c), \hat{\beta}(c))$ (i.e.~$V_c\geq G_c$) if and only if $\EE[\int_0^{\tau_{\gamma,\beta}}\lambda e^{-\lambda t} X^x_t dt] \leq 0$ for $x\in (\hat{\gamma}(c), \hat{\beta}(c))$.
Thanks to 
Dynkin's and Green's formulae (cf.~also \cite{DayKar}, eq.~(4.3))
\begin{align}
\label{strongMarkov}
\EE&\bigg[\int_0^{\tau_{\gamma,\beta}}\lambda e^{-\lambda t} X^x_t dt\bigg]= x - \EE\big[e^{-\lambda \tau_{\gamma,\beta}} X^x_{\tau_{\gamma,\beta}}\big] \nonumber\\
& =  x - \Big\{\hat{\gamma}(c)\EE\big[e^{-\lambda \tau_\gamma}\mathds{1}_{\{\tau_\gamma<\tau_\beta\}}\big] + \hat{\beta}(c)\EE\big[e^{-\lambda \tau_\beta}\mathds{1}_{\{\tau_\beta<\tau_\gamma\}}\big]\Big\} \nonumber \\
& = x - \bigg\{\hat{\gamma}(c) \frac{\sinh\big(\sqrt{2\lambda}(\hat{\beta}(c) - x)\big)}{\sinh\big(\sqrt{2\lambda}(\hat{\beta}(c) - \hat{\gamma}(c))\big)} + \hat{\beta}(c)\frac{\sinh\big(\sqrt{2\lambda}(x - \hat{\gamma}(c))\big)}{\sinh\big(\sqrt{2\lambda}(\hat{\beta}(c)- \hat{\gamma}(c))\big)}\bigg\} \nonumber\\
& =\frac{1}{\sinh\big(\sqrt{2\lambda}(\hat{\beta}(c)- \hat{\gamma}(c))\big)} \Theta(x,c; \hat{\beta}(c), \hat{\gamma}(c)),
\end{align}
where we define
\begin{align}
\label{def-vartheta}
&\Theta(x,c; \hat{\gamma}(c), \hat{\beta}(c))\\
&:=\Big[x\sinh\big(\sqrt{2\lambda}(\hat{\beta}(c)- \hat{\gamma}(c))\big) - \hat{\gamma}(c) \sinh\big(\sqrt{2\lambda}(\hat{\beta}(c) - x)\big) - \hat{\beta}(c)\sinh\big(\sqrt{2\lambda}(x - \hat{\gamma}(c))\big)\Big].\nonumber
\end{align}
To simplify notation we set $\vartheta(x,c):=\Theta(x,c; \hat{\gamma}(c), \hat{\beta}(c))$.
The right-hand side of \eqref{strongMarkov} is negative for any $x \in (\hat{\gamma}(c), \hat{\beta}(c))$ if and only if $\vartheta(x,c) \leq 0$ therein. To study the sign of $\vartheta$ we first note that $\vartheta(\hat{\gamma}(c),c)=0=\vartheta(\hat{\beta}(c),c)$ and
\begin{equation}
\label{derivativesvartheta}
\left\{
\begin{array}{l}
\vartheta_x(x,c) = \sinh\big(\sqrt{2\lambda}(\hat{\beta}- \hat{\gamma})(c)\big)\\[+3pt]
\hspace{+1.8cm}+ \sqrt{2\lambda}\Big[\hat{\gamma}(c)\cosh\big(\sqrt{2\lambda}(\hat{\beta}(c)- x)\big) - \hat{\beta}(c)\cosh\big(\sqrt{2\lambda}(x -\hat{\gamma}(c))\big)\Big] \\[+10pt]
\vartheta_{xx}(x,c)= -2\lambda\hat{\gamma}(c)\sinh\big(\sqrt{2\lambda}(\hat{\beta}(c)- x)\big) - 2\lambda\hat{\beta}(c)\sinh\big(\sqrt{2\lambda}(x- \hat{\gamma}(c))\big)\\[+4pt]
\vartheta_{xxx}(x,c)=2\lambda\sqrt{2\lambda}\Big[ \hat{\gamma}(c)\cosh\big(\sqrt{2\lambda}(\hat{\beta}(c)- x)\big) - \hat{\beta}(c)\cosh\big(\sqrt{2\lambda}(x- \hat{\gamma}(c))\big)\Big].
\end{array}
\right.
\end{equation}
From \eqref{derivativesvartheta} it is easy to see that $i)$ $\vartheta_x(\hat{\gamma}(c),c) < 0$, since $\hat{\gamma}(c) \leq -1/\sqrt{2\lambda}$, $ii)$ $\vartheta_{xx}(\hat{\gamma}(c),c) > 0$, $\vartheta_{xx}(\hat{\beta}(c),c) < 0$ and $iii)$ $\vartheta_{xxx}(x,c)<0$. Hence $x \mapsto \vartheta_{xx}(x,c)$ is strictly decreasing and there exists a unique point $x_*:=x_*(c)$ such that $\vartheta_{xx}(x_*,c) = 0$. Clearly $x_*$ is a maximum of $x \mapsto \vartheta_x(x,c)$ in $(\hat{\gamma}(c), \hat{\beta}(c))$. We claim now, and will prove later, that $\vartheta_{x}(\hat{\beta}(c),c) > 0$. Then $\vartheta_x(x,c)>0$ for $x\in(x_*,\hat{\beta}(c))$. Moreover since $\vartheta_x(\hat{\gamma}(c),c) < 0$, there exists a unique point $x_*':=x_*'(c) < x_*$ such that $\vartheta_x(x_*',c) = 0$. This point $x_*'$ is the unique stationary point of $\vartheta(\,\cdot\,,c)$ in $(\hat{\gamma}(c),\hat{\beta}(c))$ and it is a negative minimum due to the fact that $\vartheta_{xx}(x,c) > 0$ for any $x < x_*$. Therefore, recalling also $\vartheta(\hat{\gamma}(c),c)=0=\vartheta(\hat{\beta}(c),c)$, we conclude that $\vartheta(x,c) < 0$ for any $x \in (\hat{\gamma}(c),\hat{\beta}(c))$. From \eqref{FeyKac-u} and \eqref{strongMarkov} we thus get $u(x,c) \geq 0$ for any $x \in (\hat{\gamma}(c),\hat{\beta}(c))$.

To complete the proof it remains to show that $\vartheta_{x}(\hat{\beta}(c),c) > 0$. For that it is convenient to rewrite the first equation of \eqref{derivativesvartheta} in terms of $\hat{y}_1(c)$ and $\hat{y}_2(c)$ (cf.\ \eqref{def:gammabeta}) so to have
\begin{eqnarray}
\label{thetaprimebeta}
& & \vartheta_x(\hat{\beta}(c),c)= \Theta_x(F^{-1}_{\lambda}(\hat{y}_2(c)),c; F^{-1}_{\lambda}(\hat{y}_1(c)), F^{-1}_{\lambda}(\hat{y}_2(c)))\nonumber \\
& & = \hat{y}_2^{\frac{1}{2}}(c)\hat{y}_1^{-\frac{1}{2}}(c)\big(1 - \tfrac{1}{2}\ln\hat{y}_2(c)\big) - \hat{y}_2^{-\frac{1}{2}}(c)\hat{y}_1^{\frac{1}{2}}(c)\big(1 + \tfrac{1}{2}\ln\hat{y}_2(c)\big).
\end{eqnarray}
From system \eqref{system-y1y2} (see also \eqref{system-y1y2-tris}, \eqref{F1} and \eqref{F2}) we obtain
\begin{equation}
\label{fromsystem}
\left\{
\begin{array}{l}
\displaystyle\hat{y}_2^{\frac{1}{2}}(c)\big(1 - \tfrac{1}{2}\ln\hat{y}_2(c)\big) = \frac{-2e^{-1}R(\hat{c}) + R(c)\hat{y}_1^{\frac{1}{2}}(c)\big(1 - \tfrac{1}{2}\ln\hat{y}_1(c)\big)}{R(c) - R(\hat{c})}\\[+8pt]
\displaystyle\hat{y}_2^{-\frac{1}{2}}(c)\big(1 + \tfrac{1}{2}\ln\hat{y}_2(c)\big) = \frac{R(c)}{R(c) - R(\hat{c})}\hat{y}_1^{-\frac{1}{2}}(c)\big(1 + \tfrac{1}{2}\ln\hat{y}_1(c)\big),
\end{array}
\right.
\end{equation}
which plugged into \eqref{thetaprimebeta} give
\begin{eqnarray}
\label{thetaprimebeta-bis}
& & 2\vartheta_x(\hat{\beta}(c),c)=2\Theta_x(F^{-1}_{\lambda}(\hat{y}_2(c)),c; F^{-1}_{\lambda}(\hat{y}_1(c)), F^{-1}_{\lambda}(\hat{y}_2(c)))\nonumber \\
& & = \frac{\hat{y}_1^{-\frac{1}{2}}(c)}{R(\hat{c})-R(c)}\Big[ 2e^{-1}R(\hat{c}) + R(c)\sqrt{\hat{y}_1(c)}\ln\hat{y}_1(c)\Big].
\end{eqnarray}
Recalling now that $0 < \hat{y}_1(c) < e^{-2}$, $R(\hat{c})>R(c)> 0$ and noting that the function $\sqrt{x}\ln(x)$ is nonnegative on $[0,e^{-2}]$, we conclude by \eqref{thetaprimebeta-bis} that $\vartheta_x(\hat{\beta}(c),c) > 0$ for all $c\in(c_o,\hat{c})$ as claimed.
\end{proof}

\begin{proof}{\textbf{[Proposition \ref{prop:C1past}]}}\vspace{+5pt}

Since $\hat{\beta},\,\hat{\gamma}\in C^1(c_o,\hat{c})$ and their limits exist and are finite at $\hat{c}$, one can verify by direct computation in \eqref{VDayKar-bis} (recalling also Lemma \ref{lem:Hreg} and that $V\in C^1(\RR\times(0,\hat{c}))$) that $W^1$, $W^1_c$ and $W^1_x$ are uniformly continuous on open sets of the form $(-R,R)\times(\delta,\hat{c})$ for $\delta>0$ and arbitrary $R>0$. Therefore $W^1$ has a $C^1$ extension to $\RR\times(0,\hat{c}]$ which we denote again by $W^1$.

For $x\in(-\infty, \gamma^o]\cup[\hat{\beta}(\hat{c}-),+\infty)$ we have $W^1(x,\hat{c})=W^o(x,\hat{c})$, $W^1_c(x,\hat{c})=W^o_c(x,\hat{c})$ and $W^1_x(x,\hat{c})=W^o_x(x,\hat{c})$ since $V=G$, $V_c=G_c$ and $V_x=G_x$ in that set (cf.~\eqref{def-V}, \eqref{def-G} and \eqref{Wo}).
For $x\in(\gamma^o,\hat{\beta}(\hat{c}-))$ we have
\begin{align}
\label{eq:W1-00} W^1(x,\hat{c})=&x\Phi(\hat{c})+\phi_\lambda(x)Q(F_\lambda(x),\hat{c}-)\\[+4pt]
\label{eq:W1-01} W^1_c(x,\hat{c})=&x\Phi'(\hat{c})+\phi_\lambda(x)Q_c(F_\lambda(x),\hat{c}-)\\[+4pt]
\label{eq:W1-02} W^1_x(x,\hat{c})=&\Phi(\hat{c})+\phi_\lambda(x)\Big[Q_x(F_\lambda(x),\hat{c}-)F_\lambda'(x)-\sqrt{2\lambda}Q(F_\lambda(x),\hat{c}-)\Big]
\end{align}
by \eqref{Wstar2} and Proposition \ref{prop:DayKar}. To find an explicit expression of \eqref{eq:W1-00} we study $Q(y,\hat{c}-)$ for $y\in(e^{-2},\hat{y}_2(\hat{c}-))$ (see $(1)$ of Proposition \ref{monboundaries}). In particular from \eqref{WDayKar}, Remark \ref{rem:day} and Proposition \ref{monboundaries} (noting that $\hat{y}_1(c)<e^{-2}$ for $c<\hat{c}$) we find
\begin{align}\label{Qchat}
Q(y,\hat{c}-)=H_y(e^{-2}-,\hat{c})(y-e^{-2})+H(e^{-2}-,\hat{c})=-\tfrac{1}{\sqrt{2\lambda}} R(\hat{c})e^{-1}.
\end{align}
It then follows that $W^1(x,\hat{c})=W^o(x,\hat{c})$ by simple calculations, \eqref{Wo} and \eqref{def-F}.

For \eqref{eq:W1-01} we consider $Q_c(y,\hat{c}-)$ for $y\in(e^{-2},\hat{y}_2(\hat{c}-))$ and arguing as above we obtain
\begin{align}
Q_c(y,\hat{c}-)=\Big[H_{yc}(e^{-2}-,\hat{c})+H_{yy}(e^{-2}-,\hat{c})\hat{y}_1'(\hat{c}-)\Big](y-e^{-2})+H_c(e^{-2}-,\hat{c})=0
\end{align}
by \eqref{def-H2}, \eqref{def-Hy} and \eqref{def-Hyy}, hence $V_c(x,\hat{c}-)=0$ and $W^1_c(x,\hat{c})=W^o_c(x,\hat{c})=-x$ by recalling that $\Phi'(\hat{c})=-1$ (cf.~\eqref{def-chat}).

To conclude the proof we observe that $Q_y(y,\hat{c}-)=H_y(\hat{y}_1(\hat{c}-),\hat{c})=0$ for $y\in(e^{-2},\hat{y}_2(\hat{c}-))$, hence \eqref{Qchat} and \eqref{eq:W1-02} give $W^1_x(x,c)=W^o_x(x,\hat{c})=\Phi(\hat{c})+\phi_\lambda(x)R(\hat{c})e^{-1}$.
\end{proof}


\bigskip

\textbf{Acknowledgments.} The authors thank two anonymous referees for their pertinent and useful comments.

\end{document}